\documentclass[10pt]{article}
\usepackage[T1]{fontenc}
\usepackage{amsmath,amsfonts,amsthm,mathrsfs,amssymb}
\usepackage{cite,bm}
\usepackage{multirow}
\usepackage{graphicx,float}
\usepackage{subfigure}
\usepackage{placeins}
\usepackage{color}
\usepackage{indentfirst}
\numberwithin{equation}{section}
\DeclareMathOperator{\diag}{diag}

\newcommand{\R}{{\mathbb R}}

\newcommand{\C}{{\mathbb C}}

\newcommand{\be}{\begin{eqnarray}}
\newcommand{\ben}{\begin{eqnarray*}}
\newcommand{\en}{\end{eqnarray}}
\newcommand{\enn}{\end{eqnarray*}}

\newcommand{\ov}{\overline}
\newcommand{\curl}{{\rm curl\,}}

\newcommand{\divv}{{\rm div\,}}
\newcommand{\real}{{\rm Re\,}}
\newcommand{\Ima}{{\rm Im\,}}

\newtheorem{theorem}{Theorem}[section]
\newtheorem{lemma}[theorem]{Lemma}

\newtheorem{remark}[theorem]{Remark}

\definecolor{rot}{rgb}{1.000,0.000,0.000}
\definecolor{rot1}{rgb}{0.000,0.000,0.000}
\begin{document}
\renewcommand{\theequation}{\arabic{section}.\arabic{equation}}
\begin{titlepage}
\title{PML-based boundary integral equation method for electromagnetic scattering problems in a layered-medium}

\author{
Gang Bao\thanks{School of Mathematical Sciences, Zhejiang University, Zhejiang 310027, China. Email:{\tt baog@zju.edu.cn}},
Wangtao Lu\thanks{School of Mathematical Sciences, Zhejiang University, Zhejiang 310027, China. Email:{\tt wangtaolu@zju.edu.cn}},
Tao Yin\thanks{LSEC, Institute of Computational Mathematics and Scientific/Engineering Computing, Academy of Mathematics and Systems Science, Chinese Academy of Sciences, Beijing 100190, China. Email:{\tt yintao@lsec.cc.ac.cn}},
Lu Zhang\thanks{School of Mathematical Sciences, Zhejiang University, Hangzhou 310027, China. Email:{\tt zhanglu0@zju.edu.cn}}}
\date{}
\end{titlepage}
\maketitle

\begin{abstract}
This paper proposes a new boundary integral equation (BIE) methodology based on the perfectly matched layer (PML) truncation technique for solving the electromagnetic scattering  problems in a multi-layered medium. Instead of using the original PML stretched fields, artificial fields which are also equivalent to the solutions in the physical region are introduced. This significantly simplifies the study of the proposed methodology to derive the PML problem. Then some PML transformed layer potentials and the associated boundary integral operators (BIOs) are defined and the corresponding jump relations are shown. Under the assumption that the fields vanish on the PML boundary, the solution representations, as well as the related BIEs and regularization of the hyper-singular operators, in terms of the current density functions on the truncated interface, are derived. Numerical experiments are presented to demonstrate the efficiency and accuracy of the method.

\end{abstract}

\section{Introduction}
	\label{sec:1}
	Scattering problems in layered media have important applications in diverse scientific areas such as radio communications, remote subsurface sensing, and plasmonics, to name a few. The BIE method offers an attractive approach for solving such problems in an unbounded domain and it possesses such advantages as avoiding dispersion errors, while using high-order discretization method, and only requiring discretization of the boundaries~\cite{BY20,BY21,DM97,HW08,MC96,TCPS02}. For large scale problems, the resulted dense linear systems produced by the BIEs can be efficiently solved by means of iterative linear algebra solvers in conjunction with adequate acceleration techniques~\cite{BB21,GJ23,GR87}.
	
	One of the most popular BIE methods for solving the layered-medium scattering problems typically uses layer Green function (LGF) which can automatically reduce the scattering problems to integral equations only on the local defects of the infinite plane~\cite{CY00,O04,OC04,PGM00}. However, the numerical evaluation of LGF and their derivatives requires to treat highly oscillatory integrals over infinite interval and thus, entails significant computational costs. Alternatively, the free-space Green function (FGF) can also be utilized to derive the BIEs for the layered-medium problems and the resulted BIEs, however, are established on the unbounded interfaces whose numerical solver also requires appropriate truncation strategy, for example, the approximate truncation method~\cite{MC01,SS11}, the taper function method~\cite{MSS14,SSS08,ZLSC05}, which generally suffer from low accuracy, and the windowed Green function (WGF) method~\cite{BLPT16,BP17,BY21}. Recently, a second-order method of moments (MoM) based on the WGF method have been proposed in~\cite{AP22} for the two-layered medium electromagnetic problems and a large truncated interface is necessary to achieve high accuracy.
	
	The PML truncation has been shown to be a very effective technique to truncate the problems in unbounded domains and has been extensively studied especially for domain discretization type methods, like the finite element method. The well-posedness and exponential convergence of solutions to the PML truncated problems associated with the acoustic and electromagnetic scattering problems in a two-layered medium has been proved in~\cite{CZ10,CZ17,DJZ20}. Inspired by the rapid decay of the solutions in the PML region, PML-based BIE methods have been proposed for solving the two-dimensional layered-medium problems~\cite{GL22,LLQ18,L21,YHLR22} and water-waves problem~\cite{ALC23}. This method requires to treat the integral operators involving PML stretched FGF which will cause some existed numerical discretization methodology, for example the kernel-splitting based Nystr\"om method~\cite{CK98}, to fail. In~\cite{LXYZ23}, a uniform fast and highly accurate numerical solver is developed for solving both the two- and three-dimensional acoustic scattering problems in a layered medium. But the application of this methodology to the more complicated electromagnetic scattering problems is significantly non-trivial since unlike the acoustic case~\cite{LS01}, even the jump relations associated with the PML-transformed layer potentials and BIOs for the electromagnetic problems are unknown.
	
	This work devotes to proposing a novel PML-based BIE method, as well as the mathematical foundations including the jump relations, solution representations and regularization of hyper-singular integral operators, for solving the electromagnetic scattering problems in a two-layered medium ($N=2$) and discussing its application to the problems in arbitrary multi-layered medium ($N>2$). Relying on the PML truncation strategy, the original scattering problem is truncated onto a bounded domain. The uniaxial PML stretching is applied in this work, for simplicity, which can simplify the derivation of the jump relations and regularization of the PML-transformed integral operators, but the corresponding derivations can be extended to the case of spherical PML stretching in an analogous manner. Unlike using the original PML stretched fields as that studied for the acoustic problem~\cite{LXYZ23}, artificial fields, which are still equivalent to the exact fields in the physical region, are introduced for the electromagnetic case to obtain the PML truncated problem. In terms of the PML stretched FGF of the Helmholtz equation, we define two new PML-transformed layer potentials, and the associated two BIOs, to derive the solutions representations and rigorously show the jump relations between them. Owing to the use of artificial fields, the resulted jump relations, see Theorem~\ref{lemma1}, are just like the well-known results for the classical electromagnetic problems in free-space. In addition, a regularization of the PML-transformed hyper-singular BIO is proved to re-express the hyper-singular operator as a combination of weakly-singular integral operators and surface differential operators. Then the Chebyshev-based rectangular-polar solver developed in~\cite{BG20,BY20,LXYZ23} is utilized for the numerical implementation of the BIEs. Various numerical examples, considering the half-space problems and two/three-layered medium problems, are presented to demonstrate the superior performance of the proposed method. We believe that the developed numerical solver can also be extended to the more challenging elastic layered-medium scattering problems, fluid-solid interaction problems and electromagnetic-elastic coupled problems, although the convergence theory of the PML truncation remains open, and these will be left for future works.
	
	This paper is organized as follows. In Section~\ref{sec2}, we describe the electromagnetic scattering problems in a two-layered medium and, as a comparison to the new results, briefly present the classical layer potentials, BIOs and jump relations. Section~\ref{sec3.1} introduces the used PML stretching and the truncated PML problem in terms of artificial fields is obtained in Section~\ref{sec3.2}. Relying on the jump relations associated with the newly defined PML-transformed layer potentials and BIOs, Section~\ref{sec3.3} establishes the PML-based BIEs for solving the two-layered medium electromagnetic problems. Section~\ref{sec3.4} derives a new regularized formulation for the hyper-singular BIO. In Section~\ref{sec4}, an application of the PML-BIE solver to arbitrary multi-layered medium scattering problems are discussed. Numerical examples considering the half-space problems and two/three-layered medium problems are presented in Section~\ref{sec5} to illustrate the performance of the proposed method.

	\section{Preliminaries}
	\label{sec2}
	\subsection{Two-layered medium problems}
	\label{sec2.1}
	
	\begin{figure}[htb]
		\centering
		\begin{tabular}{cc}
			\includegraphics[scale=0.13]{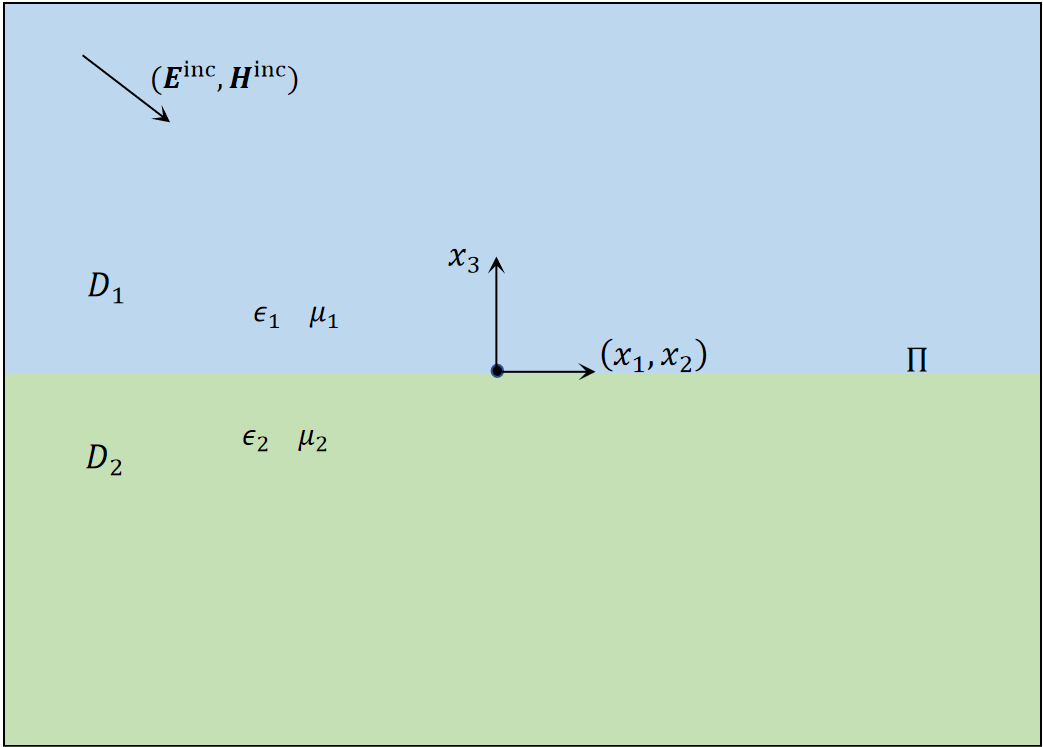} &
			\includegraphics[scale=0.13]{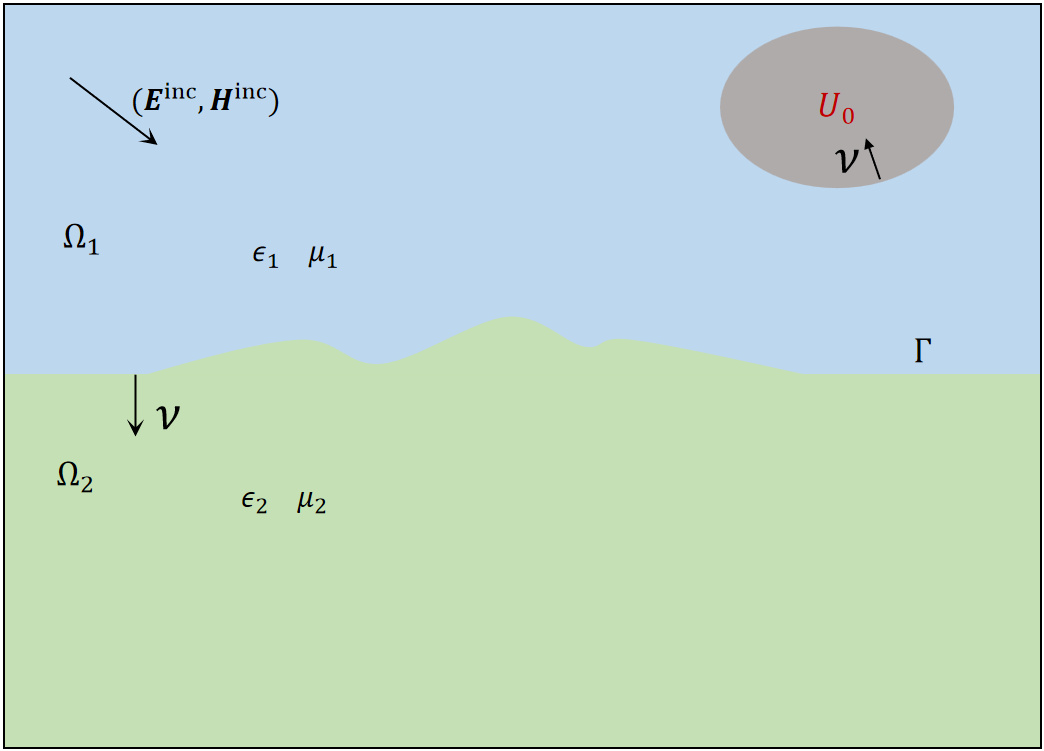} \\
			(a) Planar layered medium&(b) Perturbed planar layered medium
		\end{tabular}
		\caption{Geometry description of the problem under consideration: scattering by a planar layered medium (a) or a defect on a penetrable layer (b). $\Pi$ and $\Gamma$ denote the interface between the two-layered medium.}
		\label{Geometrydescription}
	\end{figure}
	
	As shown in Fig.~\ref{Geometrydescription}(b), an incoming wave is scattered by the dielectric interface $\Gamma$  which contains a local defect of the flat interface $\Pi$ (see Fig.~\ref{Geometrydescription}(a)) and/or a bounded obstacle $U_0$ in the presence of a two-layered medium $\Omega_j$, $j=1,2$, which is occupied by homogeneous materials characterized by the dielectric permittivity $\epsilon_j$ and the magnetic permeability $\mu_j$. For simplicity, we assume that $\Gamma$ is smooth. The unperturbed configuration depicted in Fig.~\ref{Geometrydescription}(a) for reference is composed by two planar layers $D_j$, $j=1,2$, with the interface $\Pi$ between the layers.
	
	Letting $\omega$ denote the angular frequency, the wavenumbers $k_j$ in the two layers are given by $k_j=\omega\sqrt{\mu_j\epsilon_j}, j=1,2$. In this paper, we consider two kinds of time-harmonic incident fields. The first one is a plane wave
	\be
	\label{PPW}
	\bm E^{\rm inc}(\bm x)=( \bm p\times \bm k)e^{i \bm k\cdot \bm x}\quad {\rm and} \quad \bm {H}^{\rm inc}(\bm x)=\frac{1}{\omega\mu_1}\bm k\times\bm E^{\rm inc}(\bm x),\quad \bm x\in\R^3,
	\en
	and the second one is a dipole point source
	\be
	\label{PSW}
	\begin{cases}
		\bm E^{\rm inc}(\bm x)=i\omega\mu_1\nabla_{\bm x}\times\nabla_{\bm x}\times\left[\Phi(k_1,\bm x,\bm z)\bm p\right],\cr
		\bm H^{\rm inc}(\bm x)=\nabla_{\bm x}\times\left[\Phi(k_1,\bm x,\bm z)\bm p\right],
	\end{cases} \bm x\in\R^3,
	\en
	located at $\bm z\in\Omega_1$ with $\bm x\ne \bm z$ where $\Phi(k,\bm x,\bm z)$ denotes the fundamental solution of the Helmholtz equation
	\ben
	\Phi(k,\bm x,\bm z)=\frac{\exp(ik\left|\bm x-\bm z\right|)}{4\pi\left|\bm x-\bm z\right|},
	\enn
	and $\bm p=(p_{{x}_1},p_{{x}_2},p_{{x}_3})$ and $\bm k=(0,k_{1, x_2},-k_{1,x_3})$ are constant vectors with $k_{1,x_3}> 0$ and $|\bm k|=k_1$. For simplicity, in this work we only consider $U_0=\emptyset$ for the multi-layered medium problems. Then the scattered fields $({\bm E}^{\rm sca}_j,{\bm H}^{\rm sca}_j)$, $j=1,2$ satisfy the following model of Maxwell's equations:
	\be
	\label{OP}
	\begin{cases}
		\curl {\bm E}_j^{\rm sca}-i\omega\mu_j{\bm H}_j^{\rm sca}=\bm 0 & \mbox{in}\quad\Omega_j,\cr
		\curl {\bm H}_j^{\rm sca}+i\omega\epsilon_j{\bm E}_j^{\rm sca}=\bm 0 & \mbox{in}\quad\Omega_j,\cr
		{\bm\nu}\times {\bm E}_2^{\rm sca}-{\bm \nu}\times {\bm E}_1^{\rm sca}=\bm f& \mbox{on}\quad\Gamma,\cr
		{\bm \nu}\times {\bm H}_2^{\rm sca}-{\bm \nu}\times {\bm H}_1^{\rm sca}=\bm g& \mbox{on}\quad\Gamma,\cr
		\lim\limits_{|\bm x|\to\infty}\left({\bm E}_j^{\rm sca}(\bm x)\times \bm x+|\bm x|\sqrt{\frac{\mu_j}{\epsilon_j}}{\bm H}_j^{\rm sca}(\bm x)\right)=\bm 0, &\cr
		\lim\limits_{|\bm x|\to\infty}\left({\bm H}_j^{\rm sca}(\bm x)\times \bm x-|\bm x|\sqrt{\frac{\epsilon_j}{\mu_j}}{\bm E}_j^{\rm sca}(\bm x)\right)=\bm 0, &
	\end{cases}
	\en
	where the tangential vector fields $\bm f$ and $\bm g$ are given by
	\ben
	\bm f=\bm \nu\times {\bm E}_1^{\rm src}|_{\Gamma}-\bm\nu\times {\bm E}_2^{\rm src}|_{\Gamma},\quad \bm g=\bm\nu\times \bm H_1^{\rm src}|_{\Gamma}-\bm\nu\times \bm H_{2}^{\rm src}|_{\Gamma}.
	\enn
	Here, $\bm\nu$ denotes the unit normal vector on the boundary, see Fig.~\ref{Geometrydescription}(b) and $\bm E_j^{\rm src}, \bm H_j^{\rm src}$ denote the auxiliary fields given by
	\be
	\label{auxpl}
	&& {\rm Plane \; wave\; case:}\;(\bm E_j^{\rm src},\bm H_j^{\rm src})=\begin{cases}
		(\bm E^{\rm inc},\bm H^{\rm inc})+(\bm E_1^{\rm re},\bm H_1^{\rm re}), & j=1,\cr
		(\bm E_2^{\rm re},\bm H_2^{\rm re}),& \rm j=2,
	\end{cases}\\
	\label{auxpo}
	&& {\rm Point\; source\; case:}\;(\bm E_j^{\rm src},\bm H_j^{\rm src})=\begin{cases}
		(\bm E^{\rm inc},\bm H^{\rm inc}), &  j=1,\cr
		(\bm 0,\bm 0),& \rm j=2,
	\end{cases}
	\en
	where $(\bm E_j^{\rm re},\bm H_j^{\rm re}), j=1,2$ represent the reflected fields resulting from the scattering of the plane wave $(\bm E^{\rm inc}, \bm H^{\rm inc})$ taking the form (\ref{PPW}) by the planar layer medium, see Fig.~\ref{Geometrydescription}(a). The specific forms of $(\bm E_j^{\rm src},\bm H_j^{\rm src}), j=1,2$ will be shown in Appendix.
	
	\begin{remark}
		\label{remark-half0}
		The half-space problems with $U_0$ and $\Omega_2$ being perfect electric conductor (PEC) or perfect magnetic conductor (PMC) for which the boundary condition at the boundary of the impenetrable objects are given by
		\be
		\label{PECC0}
		&&{\rm PEC\;\; case:}\qquad\bm \nu\times \bm E=0 \quad\mbox{on}\quad\Gamma_*:=\partial U_0\cup\Gamma,\\
		\label{PMCC0}
		&&{\rm PMC\;\; case:}\qquad\bm \nu\times \bm H=0 \quad\mbox{on}\quad\Gamma_*,
		\en
		respectively, can also be considered analogously. The corresponding PML-BIE will be briefly discussed in Remark~\ref{remark-halfpml}.
	\end{remark}

	\subsection{Classical layer potentials and BIOs}
	\label{sec2.2}
	
	As a comparison to Section~\ref{sec3.3}, in this section we introduce the classical layer potentials and BIOs associated with the Maxwell's equation in terms of the free-space fundamental solution $\Phi(k,\bm x,\bm y)$. For $j=1,2$, define the layer potentials and BIOs:
	\be
	\label{SLP}
	&&S_{j}(\bm \varphi)(\bm x):=\nabla_{\bm x}\times\nabla_{\bm x}\times\int_{\Gamma}\Phi(k_j,\bm x,\bm y)\bm \varphi(\bm y)ds_{\bm y},\quad \bm x\in\Omega_j,\\
	\label{DLP}
	&&D_{j}(\bm  \varphi)(\bm x):=\nabla_{\bm x}\times\int_{\Gamma}\Phi(k_j,\bm x,\bm y)\bm \varphi(\bm y)ds_{\bm y},\quad \bm x\in\Omega_j,
	\en
	and
	\be
	\label{SLO}
	N_{j}(\bm\varphi)(\bm x):=&&\bm{\nu_x}\times\nabla_{\bm x}\times\nabla_{\bm x}\times\int_{\Gamma}\Phi(k_j,\bm x,\bm y)\bm\varphi(\bm y)ds_{\bm y},\quad \bm x\in\Gamma,\\
	\label{DLO}
	K_{j}(\bm\varphi)(\bm x):=&&\bm{\nu_x}\times\nabla_{\bm x}\times\int_{\Gamma}\Phi(k_j,\bm x,\bm y)\bm\varphi(\bm y)ds_{\bm y},\quad \bm x\in\Gamma,
	\en
	for $j=1,2$. Then we conclude the following well-known jump relations.
	
	\begin{theorem}
		\label{classical-jump}
		Let $\bm\varphi\in C^1(\Gamma)^3$ be a tangential field. It holds that
		\ben
		\nabla_{\bm x}\times S_j(\bm\varphi)(\bm x)=k_j^2D_j(\bm\varphi)(\bm x),\quad \nabla_{\bm x}\times D_j(\bm\varphi)(\bm x)=S_j(\bm\varphi)(\bm x), \quad \bm x\in\Omega_j,
		\enn
		and the following jump relations hold:
		\ben
		\label{JR1}
		&&\lim_{h\to 0^+}\bm{\nu_x}\times S_1(\bm \varphi)(\bm x- h\bm{\nu_x})=N_1 (\bm \varphi)(\bm x),\quad \bm x\in\Gamma,\\
		\label{JR2}
		&&\lim_{h\to 0^+}\bm{\nu_x}\times D_1(\bm \varphi)(\bm x- h\bm{\nu_x})=(K_1-\frac{\mathbb{I}}{2})\bm \varphi(\bm x),\quad \bm x\in\Gamma,\\
		\label{JR3}
		&&\lim_{h\to 0^+}\bm {\nu_x}\times S_2(\bm \varphi)(\bm x+ h\bm{\nu_x})=N_2(\bm \varphi)(\bm x),\quad \bm x\in\Gamma,\\
		\label{JR4}
		&&\lim_{h\to 0^+}\bm{\nu_x}\times D_2(\bm \varphi)(\bm x+ h\bm{\nu_x})=(K_2+\frac{\mathbb{I}}{2})\bm \varphi(\bm x),\quad \bm x\in\Gamma.
		\enn
	\end{theorem}

	\section{The truncated PML problems and BIEs}
	\label{sec3}
	The purpose of this section is to introduce the PML truncation strategy and derive the BIEs for the truncated Maxwell's equation.
	
	\subsection{The PML stretching}
	\label{sec3.1}
	
	The PML method involves analytical stretching of the real spatial coordinates of the physical equations into the complex plane, along which the outward propagating waves must be attenuated in the absorbing layer. The well-posedness and the exponential decay of the solution in PML region has been shown in \cite{CZ17,DJZ20}. For simplicity, following \cite{LLQ18,LXYZ23}, the complex change of spatial variable $\bm x=( x_1,x_2,x_3)\in \Omega \subset \R^3\mapsto \widetilde {\bm x}(\widetilde{ x}_1,\widetilde{x}_2,\widetilde{x}_3)\in \widetilde \Omega\subset \C^3$ used in this work is defined as
	\be
	\label{CCS}
	\widetilde {x}_l({x}_l)={x}_l+i\int_0^{x_l}\sigma_l(t)dt,\quad l=1,...,3,
	\en
	for $l=1,\cdots,3$. In fact, use of this change of spatial variable ensures the resulted matrixes $\mathbb{A}, \mathbb{B}$ to be diagonal which will simplify the discussion of the jump relations and regularization of the PML transformed integral operators. Specifically, the PML medium property $\sigma_l$ are defined piecewise by the form
	\be
	\label{sigma}
	\sigma_l({x}_l)=\left\{ {\begin{array}{ll}
			\frac{2Sf_1^P}{f_1^P+f_2^P}, & a_l\le {x}_l\le a_l+T_l,\\
			S,& {x}_l>a_l+T_l,\\
			0,&-a_l<{x}_l<a_l,\\
			\sigma_l(-{x}_l),&{x}_l\le -a_l,
	\end{array}} \right.
	\en
	where $T_l>0$, $l=1,\cdots,3$, denote the thickness of the PML, $P$ ($P\ge2$) is a positive integer and
	\ben
	f_1=\left(\frac{1}{2}-\frac{1}{P}\right)\overline {x}_l^3 +\frac{\overline {x}_l}{P}+\frac{1}{2},\quad f_2=1-f_1,\quad\overline { x}_l=\frac{{x}_l-\left(a_l+T_l\right)}{T_l}.
	\enn
	Clearly $\sigma_l$ satisfy the property
	\ben
	\sigma_l(t)=\sigma_l(-t),\quad \sigma_l(t)=0\;\;\mbox{for}\;\; \left| t \right|\le a_l,\quad \sigma_l(t)>0 \;\;\mbox{for}\;\; \left| t \right|> a_l.
	\enn
	Moreover, let $B_a:=(-a_1, a_1)\times (-a_2, a_2)\times (-a_3, a_3)\subset \R^d$ be a rectangle domain. Then the infinite domain can be truncated onto a bounded domain. As shown in Fig.\ref{PML3D}, using the box $B_{a,T}:=(-a_1-T_1, a_1+T_1)\times (-a_2-T_2, a_2+T_2)\times (-a_3-T_3, a_3+T_3)$, we define the following areas
	\ben
	&&\Omega_{\rm PHY}:=B_a,\quad \Omega_{\rm PML}:=B_{a,T}\backslash\ov{B_a},\quad \Omega_1^b:=B_{a,T}\cap\Omega_1,\\
	&& \Omega_2^b:=B_{a,T}\cap\Omega_2,\quad\Gamma^b:=\Gamma\cap B_{a,T},\quad \Gamma^+:=\partial \Omega_1^b\backslash \Gamma^b,  \\
	&& \Gamma^-:=\partial \Omega_2^b\backslash \Gamma^b,\quad\Gamma_{\rm PHY}=\Gamma^b\cap\Omega_{\rm PHY},\quad\Gamma_{\rm PML}=\Gamma^b\cap\Omega_{\rm PML}.
	\enn
	
	\begin{figure}[htb]
		\centering
		\begin{tabular}{cc}
			\includegraphics[scale=0.1]{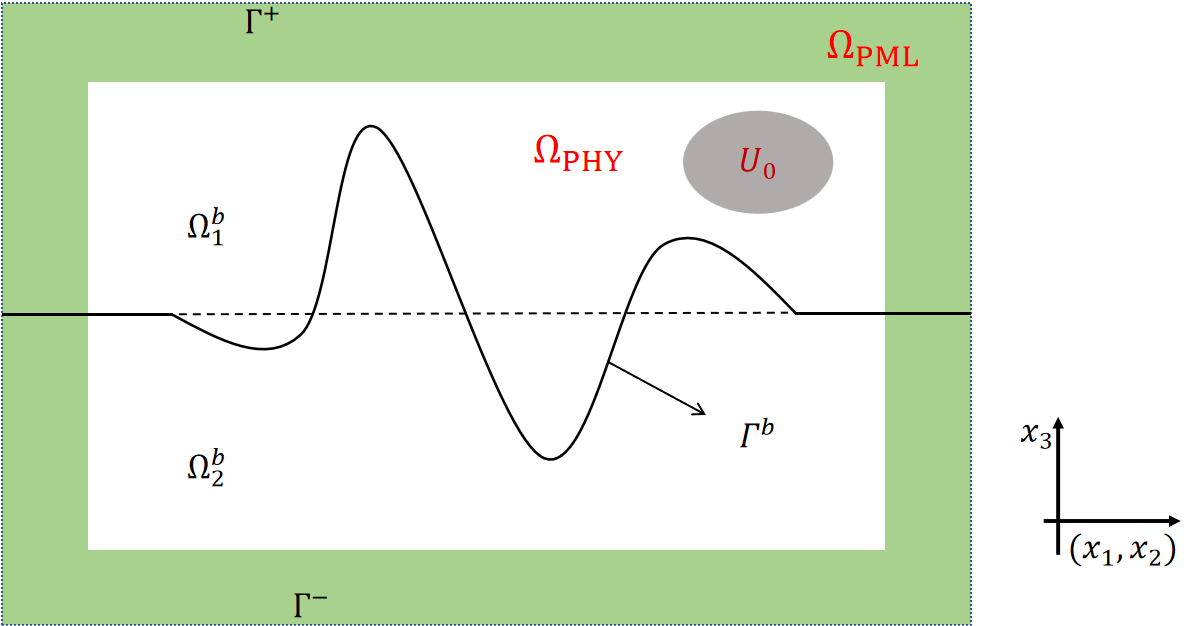} &
			\includegraphics[scale=0.12]{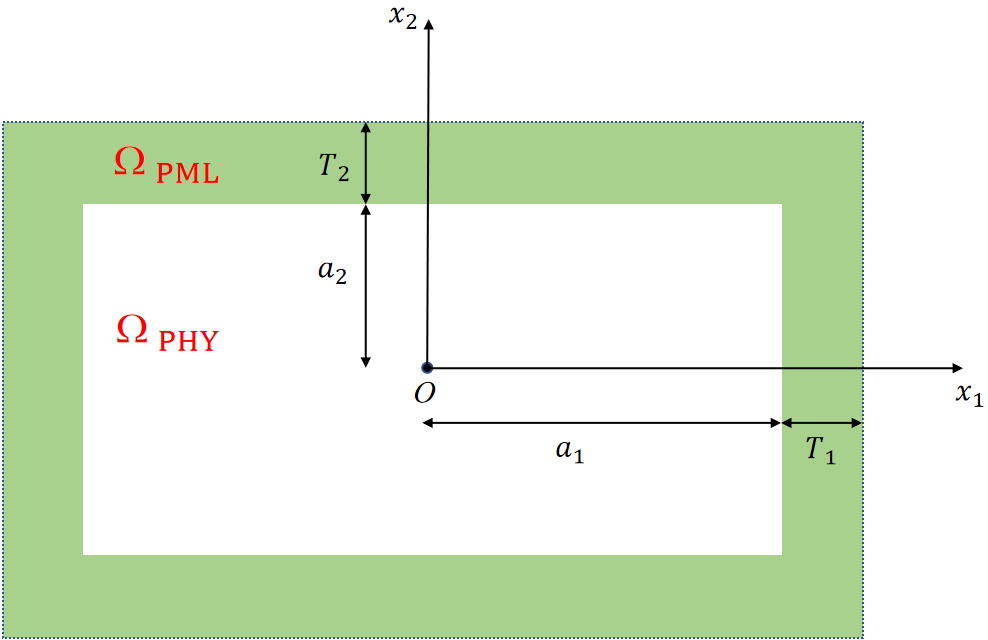}  \\
			(a)&(b)
		\end{tabular}
		\caption{(a) The PML-truncated domain; (b) The vertical view of PML truncation on the bottom infinite surface.}
		\label{PML3D}
	\end{figure}
	
	\subsection{The truncated PML problems}
	\label{sec3.2}
	
	Introducing the complex coordinate map (\ref{CCS}), a direct calculation yields the stretched gradient, curl, divergence, and Laplace operators~\cite{CZ17,LS01} as follows:
	\ben
	\label{eq1}
	&&\widetilde \nabla v:=\mathbb{B}^{-1}\nabla v,\quad \widetilde\nabla\times \bm u:=J^{-1}\mathbb{B}\curl(\mathbb{B}\bm u),\quad \widetilde\nabla\cdot \bm u:=J^{-1}{\divv}(J\mathbb{B}^{-1}\bm u),\\
	\label{eq2}
	&&\widetilde\Delta v:=J^{-1}\divv(\mathbb{A}^{-1}\nabla v),\quad \widetilde\nabla\times\widetilde\nabla\times \bm u=J^{-1}\mathbb{B}\curl\left[\mathbb{A}\curl(\mathbb{B}\bm u)\right],
	\enn
	where
	\ben
	&&\alpha_l(x_l)=1+i\sigma_l(x_l),\quad l=1,...,3,\\
	&&\mathbb{B}=\diag\left\{\alpha_1,\alpha_2,\alpha_3\right\},\quad J=\alpha_1\alpha_2\alpha_3,\quad \mathbb{A}=J^{-1}\mathbb{B}^2.
	\enn
	For the convenience of the subsequent derivation process, we introduce the artificial variables
	\ben
	\widetilde{\bm E}_j(\bm x)=\mathbb{B}(\bm x)\bm E_j(\widetilde {\bm x}),\quad \widetilde{\bm H}_j(\bm x)=\mathbb{B}(\bm x)\bm H_j(\widetilde {\bm x}).
	\enn
	It is easy to see that $\widetilde{\bm E}_j(\bm x)=\bm E_j(\bm x)$ in the physical region $\Omega_{\rm PHY}$.
	
	With the help of the above notations, we can obtain the stretched Maxwell's equations
	\be
	\label{PMLP}
	\begin{cases}
		\curl {\widetilde {\bm E}}_j^{\rm sca}-i\omega\mu_j\mathbb{A}^{-1}{\widetilde {\bm H}}_j^{\rm sca}=0 & \mbox{in}\quad\Omega_j^b,\cr
		\curl {\widetilde {\bm H}}_j^{\rm sca}+i\omega\epsilon_j\mathbb {A}^{-1}{\widetilde {\bm E}}_j^{\rm sca}=0 & \mbox{in}\quad\Omega_j^b.
	\end{cases}
	\en	
	On $\Gamma^b$, the original continuity conditions on $\Gamma$ can be transformed into
	\be
	\label{PMLB}
	\begin{cases}
		\bm{\nu}\times {\widetilde {\bm E}}_2^{\rm sca}-\bm{\nu}\times {\widetilde {\bm E}}_1^{\rm sca}=\widetilde{\bm f}& \mbox{on}\quad\Gamma^b,\cr
		\bm{\nu}\times {\widetilde {\bm H}}_2^{\rm sca}-\bm{\nu}\times {\widetilde {\bm H}}_1^{\rm sca}=\widetilde{\bm g}& \mbox{on}\quad\Gamma^b,
	\end{cases}
	\en	
	where the tangential vector fields $\widetilde{\bm f}$ and $\widetilde{\bm g}$ are given by
	\ben
	\widetilde{\bm f}=\bm {\nu}\times \widetilde{\bm E}_1^{\rm src}|_{\Gamma^b}-\bm{\nu}\times \widetilde{\bm E}_2^{\rm src}|_{\Gamma^b},\quad \widetilde{\bm g}=\bm{\nu}\times {\widetilde {\bm H}}_1^{\rm src}|_{\Gamma^b}-\bm{\nu}\times {\widetilde {\bm H}}_2^{\rm src}|_{\Gamma^b},
	\enn
	with $\widetilde{\bm E}_j^{\rm src}(\bm x)=\mathbb B(\bm x){\bm E}_j^{\rm src}(\widetilde{\bm x})$ and $\widetilde{\bm H}_j^{\rm src}(\bm x)=\mathbb B(\bm x){\bm H}_j^{\rm src}(\widetilde {\bm x})$, $j=1,2$.

	Inspired by the substantial convergence analysis in~\cite{CZ17,DJZ20} which shows that an exponentially decay truncation error will be induced due to the PML truncation, we assume that
	\be
	\label{assu}
	\begin{cases}
		\widetilde{\bm E}_1^{\rm sca}= \widetilde{\bm H}_1^{\rm sca}=\bm0 & \mbox{on}\quad \Gamma^+, \cr
		\widetilde{\bm E}_2^{\rm sca}= \widetilde{\bm H}_2^{\rm sca}=\bm0 & \mbox{on}\quad \Gamma^-.
	\end{cases}
	\en
	

	
	\subsection{BIEs}
	\label{sec3.3}

	Next, we derive the corresponding BIEs for solving the truncated PML problems. To do this, we define the following PML transformed layer potentials and BIOs
	\be
	\label{PMLSLP}
	&&\widetilde S_j(\bm\varphi)(\bm x)=\mathbb{B}(\bm x)\widetilde \nabla_{\bm x}\times\widetilde \nabla_{\bm x}\times\int_{\Gamma^b} \widetilde\Phi(k_j,\bm x,\bm y)\mathbb{B}(\bm y)\bm\varphi(\bm y)ds_{\bm y}, \bm x\in\Omega_j^b,\\
	\label{PMLDLP}
	&&\widetilde D_j(\bm\varphi)(\bm x)=\mathbb{B}(\bm x)\widetilde \nabla_{\bm x}\times\int_{\Gamma^b} \widetilde\Phi(k_j,\bm x,\bm y)\mathbb{B}(\bm y)\bm\varphi(\bm y)ds_{\bm y}, \bm x\in\Omega_j^b,\\
	\label{PMLSIO}
	&&\widetilde N_j(\bm\varphi)(\bm x)=\bm {\nu_x}\times \mathbb{B}(\bm x)\widetilde \nabla_{\bm x}\times\widetilde \nabla_{\bm x}\times\int_{\Gamma^b} \widetilde\Phi(k_j,\bm x,\bm y)\mathbb{B}(\bm y)\bm\varphi(\bm y)ds_{\bm y}, \bm x\in\Gamma^b,\\
	\label{PMLDIO}
	&&\widetilde K_j(\bm\varphi)(\bm x)=\bm{\nu_x}\times \mathbb{B}(\bm x)\widetilde \nabla_{\bm x}\times\int_{\Gamma^b} \widetilde\Phi(k_j,\bm x,\bm y)\mathbb{B}(\bm y)\bm\varphi(\bm y)ds_{\bm y}, \bm x\in\Gamma^b,
	\en
	where $\widetilde\Phi(k_j,\bm x,\bm y)$ denotes the PML-transformed free-space fundamental solution that takes the form~\cite{LLQ18,LXYZ23}
	\ben
	\widetilde\Phi(k_j,\bm x,\bm y)=\Phi(k_j,\widetilde {\bm x},\widetilde {\bm y})=\frac{{\rm exp}\left(ik\rho(\widetilde {\bm x},\widetilde{\bm y})\right)}{4\pi\rho(\widetilde {\bm x},\widetilde {\bm y})}
	\enn
	with $\rho$ denoting the complex distance function given by
	\ben
	\rho(\widetilde{\bm x},\widetilde{\bm y})=\left(\sum\limits_{j=1}^3 \left( {\widetilde{x}}_j-{\widetilde {y}}_j\right)^2 \right)^{1/2}.
	\enn
	Here, the half-power operator $z^{1/2}$ is chosen to be the branch of $\sqrt z$ with nonnegative real part for $z\in\mathcal \C\backslash \left( -\infty,  \right.\left. 0 \right]$ such that $\mathrm{arg}\sqrt{z}\in (-\frac{\pi}{2},\frac{\pi}{2}]$.

	Next, the jump relations associated with PML transformed layer potentials and BIOs and the integral representation of the solutions to the problem (\ref{PMLP})-(\ref{PMLB}) will be investigated. Before that, we introduced some auxiliary results.
	
	\begin{lemma}
		\label{The3.1}
		For any $\bm x$, $\bm y\in\Gamma^b$, there exists a positive constant $L>0$ such that
		\be
		\label{T1}
		\left|(\bm{\nu_y},\bm {\widetilde x}-\bm{\widetilde y})\right|\le L\left|\bm{x-y}\right|^2
		\en
		and
		\be
		\label{T2}
		\left|\bm{\nu_x}-\bm{\nu_y}\right|\le L\left|\bm x-\bm y\right|.
		\en
	\end{lemma}
	\begin{proof}
		From~\cite{SS112}, we know that there exists a positive constant $L_1>0$ such that
		\ben
		&&\left|(\bm{\nu_y},\bm x-\bm y)\right|\le L_1\left|\bm{x-y}\right|^2,\quad\left|\bm{\nu_x}-\bm{\nu_y}\right|\le L_1\left|\bm x-\bm y\right|
		\enn
		for $\bm x$, $\bm y\in\Gamma^b$. Using (\ref{CCS}), we can obtain that
		\ben
		\left|(\bm{\nu_y},\bm {\widetilde x}-\bm{\widetilde y})\right|\le \left(1+2S\right)\left|(\bm{\nu_y},\bm x-\bm y)\right|\le L\left|\bm{x-y}\right|^2.
		\enn
		The proof is complete.
	\end{proof}

	\begin{lemma}
		\label{The3.2}
		The following equalities hold	
		\ben
		\int_{\partial\Omega_1^b}\bm{\nu_y}^\top \mathbb A^{-1}(\bm y)\nabla_{\bm y}\widetilde\Phi_0(\bm x,\bm y)ds_{\bm y}=\begin{cases}
			-1, &  \bm x\in\Omega_1^b,\cr
			-\frac{1}{2}, & \bm x\in\Gamma^b,
		\end{cases}
		\enn
		where $\widetilde\Phi_0(\bm x,\bm y)$ denotes the Green's function of the stretched Laplace equation, i.e.,
		\ben
		\nabla_{\bm x} \cdot (\mathbb A^{-1}(\bm x)\nabla_{\bm x}\widetilde\Phi_0(\bm x,\bm y))=-\delta(\bm x-\bm y),
		\enn
		with $\delta$ being the Dirac function. 	
	\end{lemma}	
	\begin{proof}	
		Let
		\ben
		w(\bm x)&&=\int_{\partial\Omega_1^b}\bm{\nu_y}^\top \mathbb A^{-1}(\bm y)\nabla_{\bm y}\widetilde\Phi_0(\bm x,\bm y)ds_{\bm y},\quad \bm x\in \Omega_1^b.
		\enn
		Then using Green's theorem yields
		\ben
		w(\bm x)&&=\int_{\Omega_1^b}\nabla_{\bm y} \cdot (\mathbb A^{-1}(\bm y)\nabla_{\bm y}\widetilde\Phi_0(\bm x,\bm y))d{\bm y}\\
		&&=-\int_{\Omega_1^b}\delta(\bm x-\bm y)d{\bm y}=-1,\quad\bm x\in \Omega_1^b.
		\enn
		For $\bm x\in\Gamma^b$, a further application of Gauss' theorem shows that
		\ben
		\int_{\partial\Omega_1^b}\bm{\nu_y}^\top \mathbb A^{-1}(\bm y)\nabla_{\bm y}\widetilde\Phi_0(\bm x,\bm y)ds_{\bm y}=\lim_{r\to 0^+}\int_{\partial B(x,r)\cap\Omega_1^b}\bm{\nu_y}^\top \mathbb A^{-1}(\bm y)\nabla_{\bm y}\widetilde\Phi_0(\bm x,\bm y)ds_{\bm y},
		\enn
		where $\partial B(\bm x,r)$ denotes the surface of the sphere $B(\bm x,r)$ of radius $r$ centered at $\bm x$ (without loss of generality, we assume $\bm\nu_{\bm x}=(0,0,-1)^\top$), and here the unit normal vector have the form $\bm{\nu_y}=(\sin(\phi)\cos(\theta),\sin(\phi)\sin(\theta),\cos(\phi))^\top$ for $\bm y=\bm x+r(\sin(\phi)\cos(\theta), \sin(\phi)\sin(\theta), \cos(\phi))^\top \in\partial B(\bm x,r)$. For a sufficiently small radius $r$, we can specify $\phi\in[0,\frac{\pi}{2}]$ and $\theta\in[0,2\pi]$ for $\bm y\in \partial B(\bm x,r)\cap\Omega_1^b$. By (\ref{CCS}), we can obtain that
		\ben
		&&\widetilde y_1-\widetilde x_1=\int_{x_1}^{y_1}\alpha_1(s)ds=\alpha_1(x_1)r\sin(\phi)\cos(\theta)+O(r^2),\\
		&&\widetilde y_2-\widetilde x_2=\int_{x_2}^{y_2}\alpha_2(s)ds=\alpha_2(x_2)r\sin(\phi)\sin(\theta)+O(r^2),\\
		&&\widetilde y_3-\widetilde x_3=\int_{x_3}^{y_3}\alpha_3(s)ds=\alpha_3(x_3)r\cos(\phi)+O(r^2).
		\enn
		Then we have
		\ben
		&&\lim_{r\to 0^+}\int_{\partial B(x,r)\cap\Omega_1^b}\bm{\nu_y}^\top \mathbb A^{-1}(\bm y)\nabla_{\bm y}\widetilde\Phi_0(\bm x,\bm y)ds_{\bm y}=-\frac{1}{4\pi}\lim_{r\to 0^+}\int_{0}^{2\pi}\int_{0}^{\frac{\pi}{2}}d\theta d\phi\\
		&&\frac{\alpha_1(x_1)\alpha_2(x_2)\alpha_3(x_3)r^3\sin(\phi)+O(r^4)}{r^3\left[\alpha_1^2(x_1)\sin^2(\phi)\cos^2(\theta)+\alpha_2^2(x_2)\sin^2(\phi)\sin^2(\theta)+\alpha_3^2(x_3)\cos^2(\phi) \right]^\frac{3}{2}+O(r^4)} \\
		&&=-\frac{1}{4\pi}\int_{0}^{2\pi}\int_{0}^{\frac{\pi}{2}}d\theta d\phi\\
		&&\frac{\alpha_1(x_1)\alpha_2(x_2)\alpha_3(x_3)\sin(\phi)}{\left[\alpha_1^2(x_1)\sin^2(\phi)\cos^2(\theta)+\alpha_2^2(x_2)\sin^2(\phi)\sin^2(\theta)+\alpha_3^2(x_3)\cos^2(\phi) \right]^\frac{3}{2}}\\
		&&=-\frac{1}{4\pi}\left(\int_{0}^{\frac{\pi}{2}}+\int_{\frac{\pi}{2}}^{\pi}+\int_{\pi}^{\frac{3\pi}{2}}+\int_{\frac{3\pi}{2}}^{2\pi}\right)d\left(\arctan\left(\frac{\alpha_2(x_2)}{\alpha_1(x_1)}\tan(\theta)\right)\right)=-\frac{1}{2}.
		\enn
		Hence,
		\ben
		\int_{\partial\Omega_1^b}\bm{\nu_y}^\top \mathbb A^{-1}(\bm y)\nabla_{\bm y}\widetilde\Phi_0(\bm x,\bm y)ds_{\bm y}=-\frac{1}{2},\quad \bm x\in \Gamma^b.
		\enn
		The proof is complete.	
	\end{proof}

	\begin{lemma}
		\label{The3.3}
		Let $\varphi\in C^1(\Gamma^b)$ satisfy $\varphi=0$ on $\partial\Omega_1^b\backslash\Gamma^b$. Define the potentials
		\ben
		&&v_1(\varphi)(\bm z)=\int_{\Gamma^b}\bm{\nu_y}^\top\mathbb A^{-1}(\bm y)\nabla_{\bm y}\widetilde\Phi_0(\bm z,\bm y)\varphi(\bm y)ds_{\bm y},\\
		&&v_2(\varphi)(\bm z)=\int_{\Gamma^b}\bm{\nu_y}^\top\mathbb A^{-1}(\bm y)\nabla_{\bm y}\widetilde\Phi(k,\bm z,\bm y)\varphi(\bm y)ds_{\bm y}.
		\enn
		The following jump relations hold:
		\be
		\label{K0JR}
		&&\lim_{h\to 0^+,\bm z=\bm x-h\bm{\nu_x}}v_1(\varphi)(\bm z)=v_1(\varphi)(\bm x)-\frac{1}{2}\varphi(\bm x),\quad \bm x\in\Gamma^b,\\
		\label{KJR}
		&&\lim_{h\to 0^+,\bm z=\bm x-h\bm{\nu_x}}v_2(\varphi)(\bm z)=v_2(\varphi)(\bm x)-\frac{1}{2}\varphi(\bm x),\quad \bm x\in \Gamma^b.
		\en
	\end{lemma}
	\begin{proof}
		We first prove the jump relation (\ref{K0JR}). Introduce the notation $D_h=\left\{\bm z\in\Omega_1^b|\bm z=\bm x-h\bm{\nu_x},\bm x\in\Gamma^b,h>0\right\}$. By Lemma~\ref{The3.1} and the inequality
		\be
		\label{xy1}
		|\bm x-\bm y|\le |\rho(\widetilde{\bm x},\widetilde{\bm y})|\le\left(1+2S\right)|\bm x-\bm y|,
		\en
		there exists a constant $L>0$ such that
		\ben
		\left|\bm{\nu_y}^\top\mathbb A^{-1}(\bm y)\nabla_{\bm y}\widetilde\Phi_0(\bm x,\bm y)\right|=\left|\frac{(\bm{\widetilde\nu_y},\widetilde{\bm x}-\widetilde{\bm y})}{4\pi\rho(\widetilde{\bm x},\widetilde{\bm y})^3}\right|\le \frac{L}{|\bm x-\bm y|},\quad\bm x,\bm y\in\Gamma^b,\quad \bm x\ne\bm y,
		\enn
		where
		\ben \bm{\widetilde{\nu}_y}=(\alpha_2(y_2)\alpha_3(y_3)\nu_y^1,\alpha_1(y_1)\alpha_3(y_3)\nu_y^2,\alpha_1(y_1)\alpha_2(y_2)\nu_y^3)^\top,
		\enn
		i.e., $v_1(\varphi)(\bm x), \bm x\in\Gamma^b$ has a weakly singular kernel. Therefore, the integral exists for $\bm x\in\Gamma^b$ as an improper integral and represents a continuous function on $\Gamma^b$. Define
		\ben
		w(\bm x)=\int_{\partial\Omega_1^b}\bm{\nu_y}^\top\mathbb A^{-1}(\bm y)\nabla_{\bm y}\widetilde\Phi_0(\bm x,\bm y)ds_y, \quad \bm x\in\Omega_1^b.
		\enn
		Then $v_1$ can be rewritten as
		\ben
		v_1(\varphi)(\bm z)=\varphi(\bm x)w(\bm z)+u_1(\bm z),\quad \bm x\in\Gamma^b,\quad\bm z=\bm x-h\bm{\nu_x}\in D_h,
		\enn
		where
		\ben
		\label{u1}
		u_1(\bm z)=\int_{\Gamma^b}\bm{\nu_y}^\top\mathbb A^{-1}(\bm y)\nabla_{\bm y}\widetilde\Phi_0(\bm z,\bm y)\left[\varphi(\bm y)-\varphi(\bm x)\right]ds_y.
		\enn
		Therefore, in view of Lemma~\ref{The3.2}, to show (\ref{K0JR}) it suffices to prove that
		\be
		\label{u1relation}
		\lim_{h\to 0^+,\bm z=\bm x-h\bm{\nu_x}}u_1(\bm z)=u_1(\bm x),\quad \bm x\in\Gamma^b,
		\en
		uniformly on $\Gamma^b$. Let $h$ be sufficiently small which will be specified later, we get
		\ben
		|\widetilde{\bm z}-\widetilde{\bm y}|^2 \ge\frac{1}{2}\left\{|\widetilde{\bm x}-\widetilde{\bm y}|^2+|\widetilde{\bm z}-\widetilde{\bm x}|^2\right\}.
		\enn
		Then using (\ref{T1}) and (\ref{xy1}) yields
		\ben
		\left|\bm{\nu_y}^\top\mathbb A^{-1}(\bm y)\nabla_{\bm y}\widetilde \Phi_0(\bm z,\bm y)\right|\le C_0\left\{\frac{1}{|\bm z-\bm y|}+\frac{|\bm z-\bm x|}{\left[|\bm x-\bm y|^2+|\bm z-\bm x|^2\right]^{\frac{3}{2}}}\right\},
		\enn
		where $C_0>0$ is a constant. For sufficiently small $r$, we can deduce that
		\be
		\label{u1-1}
		\int_{\Gamma^b\cap B(\bm x,r)}\left|\bm{\nu_y}^\top\mathbb A^{-1}(\bm y)\nabla_{\bm y}\widetilde \Phi_0(\bm z,\bm y)\right|ds_{\bm y}&&\le C_0\left\{\int_0^rd\rho+\int_0^\infty\frac{\rho |\bm x-\bm z|d\rho}{\left(\rho^2+|\bm x-\bm z|^2\right)^{3/2}} \right\}\nonumber\\
		&& \le C_1,
		\en
		where $C_1>0$ is a constant. In addition, noting that
		\ben
		\left|\bm{\nu_y}^\top\mathbb A^{-1}(\bm y)\nabla_{\bm y}\left[\widetilde \Phi_0(\bm z,\bm y)-\widetilde \Phi_0(\bm x,\bm y)\right]\right|&&=\frac{1}{4\pi}\left|\frac{(\bm{\widetilde{\nu}_y}, \widetilde{\bm z}-\widetilde{\bm y})}{\rho(\widetilde{\bm z}-\widetilde{\bm y})^3}-\frac{(\bm{\widetilde{\nu}_y},\widetilde{\bm x}-\widetilde{\bm y})}{\rho(\widetilde{\bm x}-\widetilde{\bm y})^3}\right|\\
		&&\le C_2\frac{h}{r^3},
		\enn
		with $C_2>0$ being a constant, we have
		\be
		\label{u1-2}
		\int_{\Gamma^b\backslash B(\bm x,r)}\left|\bm{\nu_y}^\top\mathbb A^{-1}(\bm y)\nabla_{\bm y}\left[\widetilde \Phi_0(\bm z,\bm y)-\widetilde \Phi_0(\bm x,\bm y)\right]\right|ds_{\bm y}\le C_3\frac{h}{r^3}
		\en
		for some constant $C_3>0$. Therefore, combining (\ref{u1-1})-(\ref{u1-2}) implies that
		\ben
		\left|u_1(\bm z)-u_1(\bm x)\right|\le C\left\{\sup_{\bm y\in\Gamma^b\cap B(\bm x,r)}|\varphi(\bm x)-\varphi(\bm y)|+\frac{h}{r^3}\right\}
		\enn
		for some constant $C>0$. Given an arbitrary $\epsilon>0$, we can choose $r>0$ such that
		\ben
		\sup_{\bm y\in\Gamma^b\cap B(\bm x,r)}|\varphi(\bm x)-\varphi(\bm y)|\le\frac{\epsilon}{2C}
		\enn
		for all $\bm x\in\Gamma^b$. Then, taking $h=\frac{\epsilon r^3}{2C}$, we see that $|u_1(\bm z)-u_1(\bm x)|\le\epsilon$. The proof of (\ref{u1relation}) and hence, the proof of (\ref{K0JR}) is complete.
		
		It remains to show the jump relation (\ref{KJR}). For $\bm x\in\Gamma^b$, $\bm z=\bm x-h\bm{\nu_x}\in D_h$, we write
		\be
		\label{V}
		V(\varphi)(\bm z):=v_2(\varphi)(\bm z)-v_1(\varphi)(\bm z)=\int_{\Gamma^b}K(\bm z,\bm y)\varphi(\bm y)ds_{\bm y},
		\en
		where
		\ben
		K(\bm z,\bm y):&&=\bm{\nu_y}^\top\mathbb A^{-1}(\bm y)\nabla_{\bm y}\left[\widetilde\Phi(k,\bm z,\bm y)-\widetilde\Phi_0(\bm z,\bm y)\right]\\
		&&=\frac{(\bm{\widetilde\nu_y},\widetilde{\bm z}-\widetilde{\bm y})}{4\pi\rho(\widetilde{\bm z}-\widetilde{\bm y})^3}\left[e^{ik\rho(\widetilde{\bm z}-\widetilde{\bm y})}-ik\rho(\widetilde{\bm z}-\widetilde{\bm y})e^{ik\rho(\widetilde{\bm z}-\widetilde{\bm y})}-1\right],
		\enn
		which satisfies $|K(\bm z,\bm y)|\le M$ for a constant $M$. Analogous to the proof of \cite[Theorem 6.15]{K14}, it can be deduced that $V(\varphi)$ is continuous in $D_h$. Then the proof of (\ref{KJR}) is complete.
	\end{proof}	
	
	The jump relations associated with PML transformed layer potentials
	and BIOs are concluded in the following theorem which takes an analogous form as Theorem~\ref{classical-jump}.
	
	\begin{theorem}
		\label{lemma1}
		Letting $\bm \varphi \in C^1(\Gamma^b)^3$ be a tangential field satisfying $\bm\varphi=0$ on $\partial\Omega_1^b\backslash\Gamma^b$, it follows that
		\be
		\label{SD1}
		&& \nabla_{\bm x}\times \widetilde S_j(\bm\varphi)(\bm x)=k_j^2\mathbb{A}^{-1}\widetilde D_j(\bm\varphi)(\bm x),\quad\bm x\in\Omega_j^b,\\
		\label{SD2}
		&& \nabla_{\bm x}\times\widetilde D_j(\bm\varphi)(\bm x)=\mathbb{A}^{-1}\widetilde S_j(\bm\varphi)(\bm x),\quad\bm x\in\Omega_j^b,
		\en
		and the following jump relations hold:
		\be
		\label{PMLJR1}
		&&\lim_{h\to 0^+}\bm{\nu_x}\times \widetilde S_1(\bm \varphi)(\bm x- h\bm{\nu_x})=\widetilde N_1 (\bm \varphi)(\bm x),\quad \bm x\in\Gamma^b,\\
		\label{PMLJR2}
		&&\lim_{h\to 0^+}\bm{\nu_x}\times \widetilde D_1(\bm \varphi)(\bm x- h\bm{\nu_x})=(\widetilde K_1-\frac{\mathbb{I}}{2})\bm \varphi(\bm x),\quad \bm x\in\Gamma^b,\\
		\label{PMLJR3}
		&&\lim_{h\to 0^+}\bm {\nu_x}\times \widetilde S_2(\bm \varphi)(\bm x+ h\bm{\nu_x})=\widetilde N_2(\bm \varphi)(\bm x),\quad \bm x\in\Gamma^b,\\
		\label{PMLJR4}
		&&\lim_{h\to 0^+}\bm{\nu_x}\times\widetilde D_2(\bm \varphi)(\bm x+ h\bm{\nu_x})=(\widetilde K_2+\frac{\mathbb{I}}{2})\bm \varphi(\bm x),\quad \bm x\in\Gamma^b.
		\en
	\end{theorem}	
	\begin{proof}
		Noting that $\curl(\mathbb{B}\bm u) = J\mathbb{B}^{-1}\widetilde\nabla\times \bm u= \mathbb{A}\mathbb{B}^{-1} \widetilde\nabla\times \bm u$, we have
		\ben
		&&\nabla_{\bm x}\times \widetilde S_j(\bm\varphi)(\bm x)\\
		&&=\mathbb{A}(\bm x)^{-1}\mathbb{B}(\bm x)\widetilde \nabla_{\bm x}\times\widetilde \nabla_{\bm x}\times\widetilde \nabla_{\bm x}\times\int_{\Gamma^b} \widetilde\Phi(k_j,\bm x,\bm y)\mathbb{B}(\bm y)\bm\varphi(\bm y)ds_{\bm y}\\
		&&=\mathbb{A}(\bm x)^{-1}\mathbb{B}(\bm x)\widetilde \nabla_{\bm x}\times(-\widetilde\Delta_{\bm x}+\widetilde\nabla_{\bm x}\widetilde\nabla_{\bm x}\cdot)\int_{\Gamma^b} \widetilde\Phi(k_j,\bm x,\bm y)\mathbb{B}(\bm y)\bm\varphi(\bm y)ds_{\bm y}\\
		&&=k_j^2\mathbb{A}^{-1}\widetilde D_j(\bm\varphi),
		\enn
		and
		\ben
		&&\nabla_{\bm x}\times \widetilde D_j(\bm\varphi)(\bm x) =\mathbb{A}(\bm x)^{-1}\mathbb{B}(\bm x)\widetilde \nabla_{\bm x}\times\widetilde \nabla_{\bm x}\times\int_{\Gamma^b} \widetilde\Phi(k_j,\bm x,\bm y)\mathbb{B}(\bm y)\bm\varphi(\bm y)ds_{\bm y}\\
		&&=\mathbb{A}^{-1}\widetilde S_j(\bm\varphi),
		\enn
		which completes the proof of (\ref{SD1}) and (\ref{SD2}).
		
		Next, we show (\ref{PMLJR2}). The proof of the other three jump relations follows in an analogous manner and thus is omitted. Let $0<c_1<b_1<a_1$, $0<c_2<b_2<a_2$ and denote $\Gamma_{loc}^1=\{\bm x\in\Gamma^b: |x_1|\le c_1, |x_2|\le c_2 \}$, $\Gamma_{loc}^2=\{\bm x\in\Gamma^b: |x_1|\le b_1, |x_2|\le b_2 \}$ such that $\bm\nu_{\bm x}=(0,0,-1)^\top$ for all $\bm x\in\Gamma^b\backslash\Gamma_{loc}^1$.
		
		For the case $\bm x\in\Gamma_{loc}^2$ and $\bm z=\bm x- h\bm{ \nu_x}$ with sufficiently small $h>0$, we define
		\ben
		\bm{\nu_x}\times \widetilde D_1(\bm \varphi)(\bm z)=I_1+I_2,
		\enn
		with
		\ben
		&&I_1(\bm z)=\bm{\nu_x}\times \mathbb{B}(\bm z)\widetilde \nabla_{\bm z}\times\int_{\Gamma_{\rm PML}} \widetilde\Phi(k_1,\bm z,\bm y)\mathbb{B}(\bm y)\bm\varphi(\bm y)ds_{\bm y},\\
		&&I_2(\bm z)=\bm{\nu_x}\times  \nabla_{\bm z}\times\int_{\Gamma_{\rm PHY}} \Phi(k_1,\bm z,\bm y)\bm\varphi(\bm y)ds_{\bm y},
		\enn
		while noting that $\mathbb B=\mathbb I$ and $\widetilde\Phi(k_1,\bm x,\bm y)=\Phi(k_1,\bm x,\bm y)$ on $\Gamma_{\rm PHY}$. Since the kernel in $I_1$ is non-singular, we obtain from Theorem~\ref{classical-jump} that
		\be
		\label{local2-jump}
		&&\lim_{h\to 0^+}\bm{\nu_x}\times \widetilde D_1(\bm \varphi)(\bm z)= (\widetilde K_1-\frac{\mathbb{I}}{2})\bm \varphi(\bm x),\quad \bm x\in\Gamma_{loc}^2.
		\en
		
		For the case $\bm x\in\Gamma^b\backslash \Gamma_{loc}^2$ and $\bm z=\bm x- h\bm{ \nu_x}$ with sufficiently small $h>0$, we define
		\ben
		\bm{ \nu_x}\times \widetilde D_1(\bm \varphi)(\bm z)=I_3(\bm z)+I_4(\bm z)+I_5(\bm z),
		\enn
		where
		\be
		\label{I3}
		I_3(\bm z)&&=\int_{\Gamma^b}\begin{pmatrix}
			\frac{\partial\widetilde\Phi(k_1,\bm z,\bm y)}{\partial\widetilde z_1}\alpha_3( z_3)\alpha_2( y_2)\\
			\frac{\partial\widetilde\Phi(k_1,\bm z,\bm y)}{\partial\widetilde z_2}\alpha_1( z_1)\alpha_3(y_3)\\
			\frac{\partial\widetilde\Phi(k_1,\bm z,\bm y)}{\partial\widetilde z_3}\alpha_2( z_2)\alpha_1( y_1)
		\end{pmatrix}(\bm{ \nu_x}^\top-\bm{\nu_y}^\top)\bm \varphi(\bm y)ds_{\bm y},\\
		\label{I4}
		I_4(\bm z)&&=-\int_{\Gamma^b}\begin{pmatrix}
			{\bm{ \nu_x}^\top}\mathbb{A}_1(\bm z,\bm y)\nabla_{\bm z}\widetilde\Phi(k_1,\bm z,\bm y) \varphi_1(\bm y)\\
			{\bm{ \nu_x}^\top}\mathbb{A}_2(\bm z,\bm y)\nabla_{\bm z}\widetilde\Phi(k_1,\bm z,\bm y) \varphi_2(\bm y)\\
			{\bm{ \nu_x}^\top}\mathbb{A}_3(\bm z,\bm y)\nabla_{\bm z}\widetilde\Phi(k_1,\bm z,\bm y) \varphi_3(\bm y)
		\end{pmatrix}ds_{\bm y},\\
		\label{I5}
		I_5(\bm z)&&:=\int_{\Gamma^b}\begin{pmatrix}
			N_1(\bm z,\bm y)\varphi_3(\bm y)\\
			N_2(\bm z,\bm y)\varphi_1(\bm y)\\
			N_3(\bm z,\bm y)\varphi_2(\bm y)
		\end{pmatrix}ds_{\bm y},
		\en
		and
		\ben
		&&\mathbb{A}_1(\bm z,\bm y)=\diag\left\{\frac{\alpha_3( z_3)\alpha_2(y_2)}{\alpha_1( z_1)},\frac{\alpha_3( z_3)\alpha_1( y_1)}{\alpha_2( z_2)},\frac{\alpha_2( z_2)\alpha_1( y_1)}{\alpha_3( z_3)}\right\}\\
		&&\mathbb{A}_2(\bm z,\bm y)=\diag\left\{\frac{\alpha_3( z_3)\alpha_2( y_2)}{\alpha_1( z_1)},\frac{\alpha_1( z_1)\alpha_3( y_3)}{\alpha_2( z_2)},\frac{\alpha_1( z_1)\alpha_2( y_2)}{\alpha_3( z_3)}\right\}\\
		&&\mathbb{A}_3(\bm z,\bm y)=\diag\left\{\frac{\alpha_2( z_2)\alpha_3( y_3)}{\alpha_1( z_1)},\frac{\alpha_1( z_1)\alpha_3( y_3)}{\alpha_2( z_2)},\frac{\alpha_2( z_2)\alpha_1( y_1)}{\alpha_3( z_3)}\right\}\\
		&&N_1(\bm z,\bm y)=\frac{\partial \widetilde\Phi(k_1,\bm z,\bm y)}{\partial\widetilde z_1}{\nu_x^3}(\alpha_2( z_2)\alpha_3( y_3)-\alpha_3( z_3)\alpha_2( y_2)),\\
		&&N_2(\bm z,\bm y)=\frac{\partial \widetilde\Phi(k_1,\bm z,\bm y)}{\partial\widetilde z_2}{\nu_x^1}(\alpha_3( z_3)\alpha_1( y_1)-\alpha_1( z_1)\alpha_3( y_3)),\\
		&&N_3(\bm z,\bm y)=\frac{\partial \widetilde\Phi(k_1,\bm z,\bm y)}{\partial\widetilde z_3}{\nu_x^2}(\alpha_1( z_1)\alpha_2( y_2) -\alpha_2(z_2)\alpha_1( y_1)).
		\enn
		Noting that $\bm{\nu}=(0,0,-1)^\top$ on $\Gamma^b\backslash \Gamma_{loc}^1$, it follows from that
		\ben
		I_3(\bm z)=\int_{\Gamma_{loc}^1}\nabla_{\bm z}\Phi(k_1,\bm z,\bm y)(\bm{ \nu_x}^\top-\bm{\nu_y}^\top)\bm \varphi(\bm y)ds_{\bm y}.
		\enn
		whose kernel is smooth and thus, has no jump. Next, we study the term $I_4(\bm z)$ and only consider the first component, i.e.,
		\ben
		I_{41}(\bm z)=-\int_{\Gamma^b}{\bm{ \nu_x}^\top}\mathbb{A}_1(\bm z,\bm y)\nabla_{\bm z}\widetilde\Phi(k_1,\bm z,\bm y) \varphi_1(\bm y)ds_{\bm y},
		\enn
		which, by using the identity $\widetilde\nabla_{\bm y}\widetilde\Phi(k_1,\bm z,\bm y)=-\widetilde\nabla_{\bm z}\widetilde\Phi(k_1,\bm z,\bm y)$, can be expressed alternatively as
		\ben
		\label{I41}
		I_{41}(\bm z)=\int_{\Gamma^b}{\bm{ \nu_y}^\top}\mathbb{A}^{-1}(\bm y)\nabla_{\bm y}\widetilde\Phi(k_1,\bm z,\bm y) \varphi_1(\bm y)ds_{\bm y}+V(\varphi_1)(\bm z),
		\enn
		where
		\ben
		&&V(\varphi_1)(\bm z):=\int_{\Gamma^b}K(\bm z,\bm y)\varphi_1(\bm y)ds_{\bm y}\\
		=&&\int_{\Gamma^b}\begin{pmatrix}
			\alpha_3( z_3)\alpha_2(y_2)\nu_x^1-\alpha_3( y_3)\alpha_2(y_2)\nu_y^1\\
			\alpha_3( z_3)\alpha_1(y_1)\nu_x^2-\alpha_3( y_3)\alpha_1(y_1)\nu_y^2\\
			\alpha_2( z_2)\alpha_1(y_1)\nu_x^3-\alpha_1( y_1)\alpha_2(y_2)\nu_y^3				
		\end{pmatrix}^\top\widetilde\nabla_{\bm y}\widetilde\Phi(k_1,\bm z,\bm y)\varphi_1(\bm y)ds_{\bm y}.
		\enn
		It can be derived from (\ref{T2}) and the fact $|\alpha_i(z_i)-\alpha_i(y_i)|\le C|z_i-y_i|$ with $C>0$ being a constant that $|K(\bm z,\bm y)|\le M$ where $M>0$ is a constant, which further implies that $V(\varphi_1)$ has no jump. Therefore, applying (\ref{KJR}) we get
		\be
		\label{local2c-jump}
		\lim_{h\to 0^+,\bm z=\bm x-h\bm{\nu_x}}I_{41}(\bm z)=I_{41}(\bm x)-\frac{1}{2}\varphi_1(\bm x),\quad \bm x\in\Gamma^b\backslash \Gamma_{loc}^2.
		\en
		Finally, analogous to the proof of \cite[Theorem 6.15]{K14}, it can be deduced that $I_5(\bm z)$ also has no jump. Then we can conclude the jump relation (\ref{PMLJR2}) from (\ref{local2-jump}) and (\ref{local2c-jump}).
		The proof is complete.
	\end{proof}
	
	Now we can derive the BIE for solving the truncated PML problem (\ref{PMLP})-(\ref{PMLB}).
	
	\begin{theorem}
		Under the assumption (\ref{assu}), the fields within the layers admit the representations
		\be
		\label{solpml1}
		&&\widetilde {\bm E}^{\rm sca}_1(\bm x)=-\widetilde D_1(\widetilde {\bm M})(\bm x)-\frac{i}{\omega\epsilon_1}\widetilde S_1(\widetilde {\bm J})(\bm x),\quad \bm x\in\Omega_1^b,\\
		\label{solpml2}
		&&\widetilde {\bm H}_1^{\rm sca}(\bm x)=\frac{i}{\omega\mu_1}\widetilde S_1(\widetilde {\bm M})(\bm x)-\widetilde D_1(\widetilde {\bm J})(\bm x),\quad \bm x\in\Omega_1^b,\\
		\label{solpml3}
		&&\widetilde {\bm E}_2^{\rm sca}(\bm x)=\widetilde D_2(\widetilde {\bm M})(\bm x)+\frac{i}{\omega\epsilon_2}\widetilde S_2(\widetilde {\bm J})(\bm x),\quad \bm x\in\Omega_2^b,\\
		\label{solpml4}
		&&\widetilde {\bm H}_2^{\rm sca}(\bm x)=-\frac{i}{\omega\mu_2}\widetilde S_2(\widetilde {\bm M}_2)(\bm x)+\widetilde D_2(\widetilde {\bm J})(\bm x),\quad \bm x\in\Omega_2^b,
		\en
		in terms of the current density functions $\widetilde {\bm M}=\bm\nu\times \widetilde {\bm E}_2|_{\Gamma^b}$, $\widetilde {\bm J}=\bm\nu\times \widetilde {\bm H}_2|_{\Gamma^b}$, and the BIE
		\be
		\label{PMLBIE}
		&&(\mathbb E +\widetilde{\mathbb T})(\widetilde{\bm\phi})= \widetilde {\bm F} \quad\mbox{on}\quad\Gamma^b,
		\en
		results where the vector density functions $\widetilde{\bm\phi}$ and boundary data $\widetilde {\bm F}$ are given by
		\ben
		\widetilde{\bm\phi}=\begin{pmatrix}
			\widetilde {\bm M}\\
			\widetilde {\bm J}
		\end{pmatrix},\quad \widetilde {\bm F}=\begin{pmatrix}
			\epsilon_1\widetilde {\bm f}\\
			\mu_1\widetilde {\bm g}
		\end{pmatrix},
		\enn	
		and the matrix operator $\widetilde{\mathbb T}$ are defined by
		\be
		\label{MOPML}
		\widetilde{\mathbb T}=\begin{bmatrix}
			\epsilon_1 \widetilde K_1-\epsilon_2 \widetilde K_2 &\frac{i}{\omega}( \widetilde N_1- \widetilde N_2)\\
			\frac{i}{\omega}( \widetilde N_2- \widetilde N_1) & \mu_1 \widetilde K_1-\mu_2 \widetilde K_2
		\end{bmatrix}.
		\en
	\end{theorem}
	\begin{proof}
		Define the stretched dyadic Green's function by
		\ben
		\widetilde{\mathbb G}(k,\bm x,\bm y):=\mathbb B(\bm y)\mathbb G(k_1,\widetilde{\bm x},\widetilde{\bm y})=\mathbb B(\bm y)\left(\mathbb I+\frac{\widetilde\nabla_{\bm x}\widetilde\nabla_{\bm x}}{k^2}\right)\widetilde\Phi(k,x,y),
		\enn
		It can be seen from~\cite{CZ17} that $\widetilde{\mathbb G}$ satisfy the equation
		\ben
		\label{GS}
		\curl_{\bm y}\left[\mathbb A(\bm y)\curl_{\bm y}\widetilde{\mathbb G}(k_j,\bm x,\bm y)\right]-k^2\mathbb A^{-1}(y)\widetilde{\mathbb G}(k_j,\bm x,\bm y)=\delta(\bm x-\bm y)\mathbb B^{-1}(\bm y)
		\enn
		for $\bm y\in\Omega_j^b$. Without loss of generality, we consider only (\ref{solpml1}) and (\ref{solpml2}). Then we have for $\bm x\in \Omega_1^b$,
		\ben
		&&\widetilde {\bm E}^{\rm sca}_1(\bm x)=\mathbb{B}(\bm x)\int_{\Omega_1^b} \delta(\bm x-\bm y)\mathbb{B}^{-1}(\bm y)\widetilde {\bm E}^{\rm sca}_1(\bm y)ds_{\bm y}\\
		=&&-\mathbb{B}(\bm x)\int_{\partial\Omega_1^b} \left[\mathbb A(\bm y)\curl_{\bm y}\widetilde{\mathbb G}(k_1,\bm x,\bm y)\right]^\top\bm{\nu_y}\times\widetilde {\bm E}^{\rm sca}_1(\bm y)ds_{\bm y}\\
		&&-\mathbb{B}(\bm x)\int_{\partial\Omega_1^b}\widetilde{\mathbb G}^\top(k_1,\bm x,\bm y)\nu_{\bm y}\times\left(\mathbb A(\bm y)\curl_{\bm y}\widetilde {\bm E}^{\rm sca}_1(\bm y)\right)ds_{\bm y}\\
		=&&-\mathbb{B}(\bm x)\int_{\partial\Omega_1^b} \left[\widetilde \nabla\times\mathbb G(k_1,\widetilde{\bm x},\widetilde{\bm y})\right]^\top\mathbb B(\bm y)\bm{\nu_y}\times\widetilde {\bm E}^{\rm sca}_1(\bm y)ds_{\bm y}\\
		&&-\mathbb{B}(\bm x)\widetilde\nabla_{\bm x}\times\widetilde\nabla_{\bm x}\int_{\partial\Omega_1^b}\widetilde\Phi(k_1,\bm x,\bm y)\mathbb B(\bm y)\nu_{\bm y}\times\left(\mathbb A(\bm y)\curl_{\bm y}\widetilde {\bm E}^{\rm sca}_1(\bm y)\right)ds_{\bm y}\\
		=&&-\widetilde D_1(\widetilde {\bm M})(\bm x)-\frac{i}{\omega\epsilon_1}\widetilde S_1(\widetilde {\bm J})(\bm x).
		\enn	
		According to (\ref{PMLP}), (\ref{SD1}) and (\ref{SD2}), we get
		\ben
		\widetilde {\bm H}^{\rm sca}_1(\bm x)=&&\frac{1}{i\omega\mu_1}\mathbb A(\bm x)\curl_x\widetilde {\bm E}^{\rm sca}_1(\bm x)=\frac{i}{\omega\mu_1}\widetilde S_1(\widetilde {\bm M})(\bm x)-\widetilde D_1(\widetilde {\bm J})(\bm x),\quad \bm x\in\Omega_1^b.
		\enn
		Letting $\bm x\in\Gamma^b$, utilizing the jump relations (\ref{PMLJR1})-(\ref{PMLJR4}) and applying the continuity conditions (\ref{PMLB}), the BIE (\ref{PMLBIE}) results. The proof is complete.
	\end{proof}

	\begin{remark}
		\label{remark-halfpml}
		Considering the half-space problems mentioned in Remark~\ref{remark-half0}, the PML stretching of the boundary conditions of PEC (\ref{PECC0}) and PMC (\ref{PMCC0}) reads
		\ben
		&&{\rm PEC\;\; case:}\qquad\bm \nu\times \widetilde {\bm E}=0\quad\mbox{on}\quad\Gamma_*^b:=\partial U_0\cup\Gamma^b,\\
		&&{\rm PMC\;\; case:}\qquad\bm \nu\times \widetilde {\bm H}=0\quad\mbox{on}\quad\Gamma_*^b.
		\enn
		Then the following BIEs for solving the corresponding PML truncated problems
		\be
		\label{PMLPECBIE}
		&&{\rm PEC\;\; case:}\quad(\widetilde K_{01}+\frac{\mathbb I}{2})(\widetilde{\bm J})=-\frac{i}{\omega\mu_1}\widetilde N_{01}(\bm\nu\times\widetilde {\bm H}^{\rm src})\quad\mbox{on}\quad\Gamma_*^b,\\
		\label{PMLPMCBIE}
		&&{\rm PMC\;\; case:}\quad(\widetilde K_{01}+\frac{\mathbb I}{2})(\widetilde{\bm M})=\frac{i}{\omega\epsilon_1}\widetilde N_{01}(\bm\nu\times \widetilde {\bm E}^{\rm src})\quad\mbox{on}\quad\Gamma_*^b,
		\en
		can be obtained where the BIOs $\widetilde K_{01}$ and $\widetilde N_{01}$ are defined by (\ref{PMLSIO}) and (\ref{PMLDIO}), respectively with $\Gamma^b$ replaced by $\Gamma_*^b$.
	\end{remark}

\begin{remark}
We have to point out here that the wellposedness of the derived BIEs (\ref{PMLBIE}) and (\ref{PMLPECBIE})-(\ref{PMLPMCBIE}), which must relate to the analysis of the truncated PML problem (\ref{PMLP})-(\ref{assu}), still remains open. Unlike~\cite{CZ17}, use of the PML boundary condition (\ref{assu}) can significantly simplify the derivation and reduce the numerical evaluation of the BIE for solving the layered-medium scattering problems, however, will bring additional challenges to show the inf-sup condition. This will be left for future works.
\end{remark}
	
	\subsection{Regularization of the hyper-singular operators}
	\label{sec3.4}
	
	Note that the integral operators $\widetilde N_1$, $\widetilde N_2,$ and $\widetilde N_{01}$ in (\ref{PMLBIE}), (\ref{PMLPECBIE}) and (\ref{PMLPMCBIE}) are all hyper-singular. This section will develop efficient regularized formulations for these hyper-singular operators and the resulted new expressions only consist of weakly singular integral operators, surface vector curl operator and surface divergence operator.
	\begin{theorem}
		\label{REF}
		Let ${\bm\varphi}\in C^1(\Gamma^b)^3$ be a tangential field satisfying ${\bm\varphi}=0$ on $\partial\Gamma^b$. The hyper-singular operators $\widetilde N_{j}, j=1,2$ can be expressed alternatively as
		\ben
		\widetilde N_{j}({\bm\varphi})(\bm x)=&&k_j^2\bm\nu_x\times \mathbb{B}(\bm x)\int_{\Gamma_b}\widetilde\Phi(k_j,\bm x,\bm y) \mathbb{B}(\bm y){\bm\varphi}(\bm y)ds_{\bm y}\\
		&&+\overrightarrow{\curl}_\Gamma\int_{\Gamma_b}\widetilde\Phi(k_j,\bm x,\bm y)\divv_{\Gamma}{\bm\varphi}(\bm y)ds_{\bm y},
		\enn	
		where $\overrightarrow{\curl}_\Gamma$ and $\divv_\Gamma$ denote the surface vector curl operator and surface divergence operator, respectively, which are defined for a scalar function $u$ and a vector function $\bm v$ by the following equalities
		\ben
		\overrightarrow{\curl}_\Gamma u={\bm\nu}\times \nabla_\Gamma u,\quad \divv_\Gamma\bm v=\divv\bm v-{\bm\nu}\cdot\partial_{\bm\nu} \bm v,
		\enn	
		with the surface gradient $\nabla_\Gamma u=\nabla u-{\bm\nu}\partial_{\bm\nu} u$.	Analogous result holds for $\widetilde N_{01}$.
	\end{theorem}	
	
	\begin{proof}
		The identity $\widetilde\nabla\times\widetilde\nabla\times=-\widetilde\Delta +\widetilde\nabla\widetilde\nabla\cdot$ implies that
		\ben
		\widetilde N_{j}(\bm\varphi)(\bm x)
		=&&\bm{\nu_x}\times\mathbb B(\bm x)\left(-\widetilde\Delta_{\bm x}+\widetilde\nabla_{\bm x}\widetilde\nabla_{\bm x}\cdot\right)\int_{\Gamma^b}\widetilde\Phi(k_j,\bm x,\bm y)\mathbb B(y)\bm\varphi(\bm y)ds_{\bm y}\\
		=&&k_j^2\bm{\nu_x}\times\mathbb B(\bm x)\int_{\Gamma^b}\widetilde\Phi(k_j,\bm x,\bm y)\mathbb B(\bm y)\bm\varphi(\bm y)ds_{\bm y}\\
		&&-\bm{\nu_x}\times\mathbb B(\bm x)\widetilde\nabla_{\bm x}\int_{\Gamma^b}\nabla_{\bm y}\widetilde\Phi(k_j,\bm x,\bm y)\cdot \bm\varphi(\bm y)ds_{\bm y}\\
		=&&k_j^2\bm{\nu_x}\times\mathbb B(\bm x)\int_{\Gamma^b}\widetilde\Phi(k_j,\bm x,\bm y)\mathbb B(y)\bm\varphi(\bm y)ds_{\bm y}\\
		&&+\overrightarrow{\curl}_\Gamma\int_{\Gamma^b}\widetilde\Phi(k_j,\bm x,\bm y)\divv_{\Gamma}\bm\varphi(\bm y)ds_{\bm y}
		\enn
		in view of assumption on $\bm\varphi$ and the integration-by-part formula
		\ben
		\int_\Gamma \nabla_\Gamma u \cdot\bm v ds=- \int_\Gamma u\divv_\Gamma\bm v ds
		\enn
		for a scaler field $u$ and a vector field $\bm v$. This completes the proof.
	\end{proof}

	\begin{remark}
		Considering the BIEs (\ref{PMLPECBIE}) and (\ref{PMLPMCBIE}), the hyper-singular terms we need to deal with are $\widetilde N_{01}(\nu\times \widetilde {\bm E}^{\rm src})$ and $\widetilde N_{01}(\nu\times \widetilde {\bm H}^{\rm src})$.
		\begin{itemize}
			\item For the case of a plane incident wave, it is known that $\nu\times{\bm E}^{\rm src}=0$ and $\nu\times{\bm H}^{\rm src}=0$ on $\Pi$ and thus, $\nu\times \widetilde {\bm E}^{\rm src}=0$ and $\nu\times \widetilde {\bm H}^{\rm src}=0$ on $\Gamma^b\cap\Pi$. Hence, the assumption in Lemma~\ref{REF} is valid and therefore, the regularized formulations for $\widetilde N_{01}(\nu\times \widetilde {\bm E}^{\rm src})$ and $\widetilde N_{01}(\nu\times \widetilde {\bm H}^{\rm src})$ holds exactly.
			\item For the case of a point source, note that $({\bm E}^{\rm src},{\bm H}^{\rm src})=(-{\bm E}^{\rm inc},-{\bm H}^{\rm inc})$ is an outgoing wave which decays exponentially in the PML region as the thickness of the PML increases. Thus, the regularized formulations shown in Lemma~\ref{REF} can provide highly accurate approximations to $\widetilde N_{01}(\nu\times \widetilde {\bm E}^{\rm src})$ and $\widetilde N_{01}(\nu\times \widetilde {\bm H}^{\rm src})$.
		\end{itemize}
		For the terms $(\widetilde N_1-\widetilde N_2)(\widetilde{\bm g})$ and $(\widetilde N_2-\widetilde N_1)(\widetilde{\bm f})$ in (\ref{PMLBIE}), the regularized formulations shown in Lemma~\ref{REF} can be utilized in numerical implementation. But alternatively, noting that the kernel of $\widetilde N_1-\widetilde N_2$ is weakly-singular, the terms $(\widetilde N_1-\widetilde N_2)(\widetilde{\bm g})$ and $(\widetilde N_2-\widetilde N_1)(\widetilde{\bm f})$ can be evaluated directly without use of the regularization.
	\end{remark}

	\begin{figure}[htb]
		\centering
		\begin{tabular}{cc}
			\includegraphics[scale=0.13]{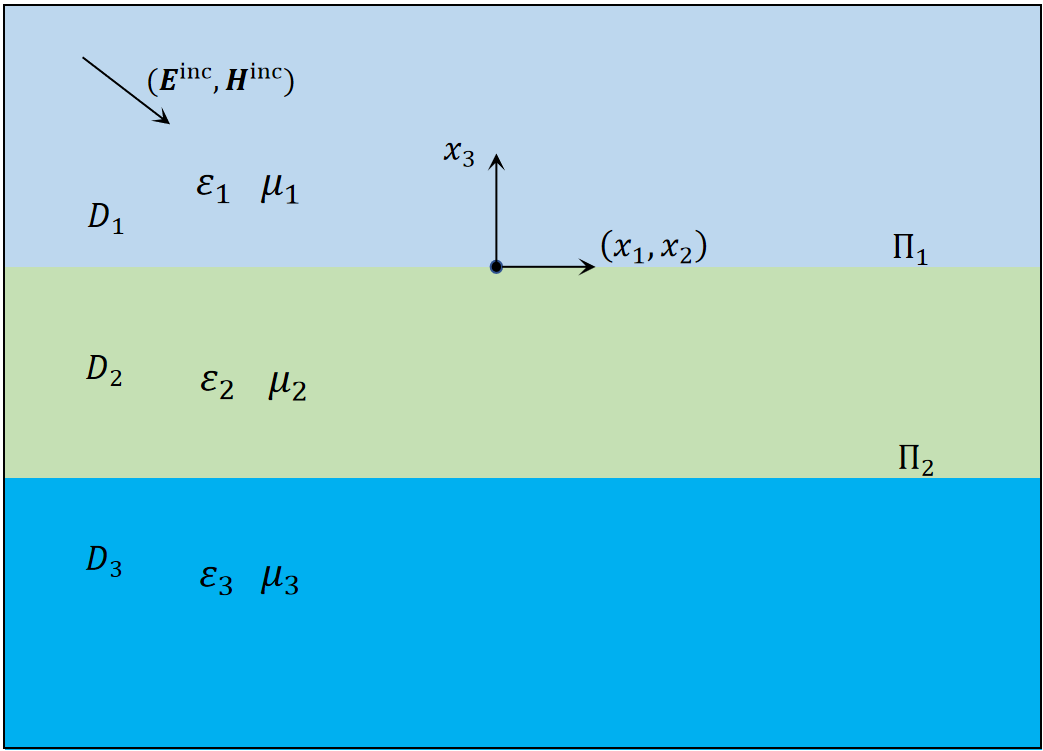} &
			\includegraphics[scale=0.13]{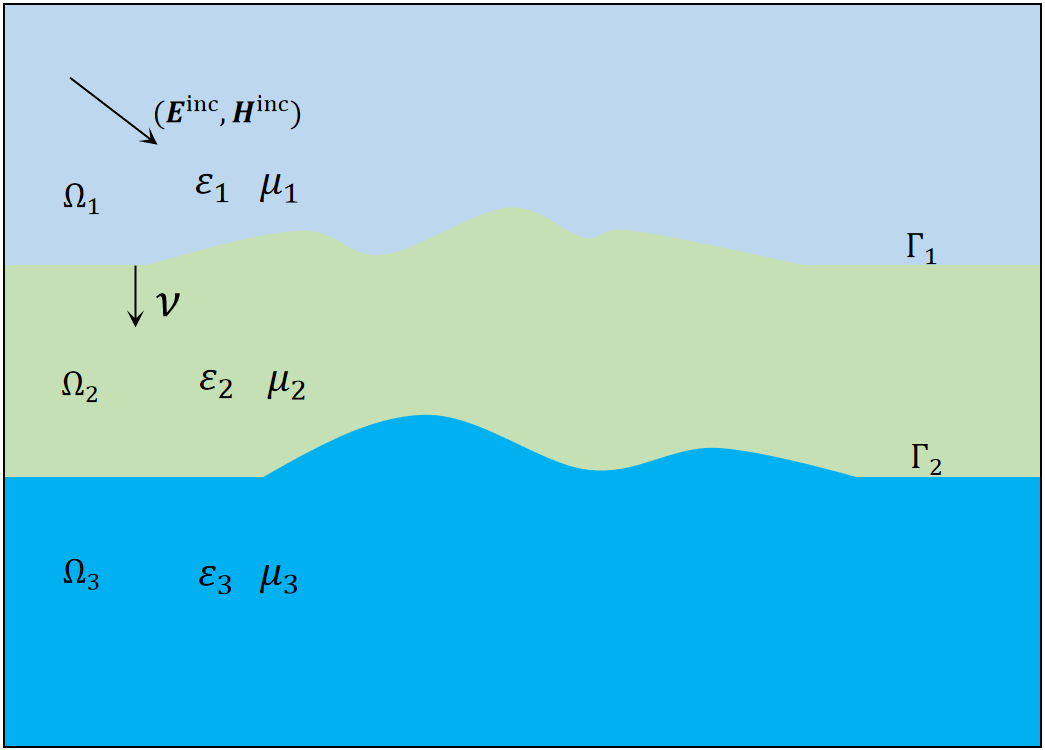} \\
			(a)&(b)
		\end{tabular}
		\caption{Geometry description of a planar layered medium (a) and a locally perturbed multi-layered medium (b) for the case $N=3$.}
		\label{Multi Geometry description}
	\end{figure}
	
	\section{Multi-layered medium problems}
	\label{sec4}
	
	In this section, we discuss the application of the PML-BIE approach to the problem of Maxwell's equations in a multi-layered medium which contains $N$ ($N>2$) layers $\Omega_j$, $j=1,...,N$, with the interfaces $\Gamma_j$, $j=1,...,N-1$, see Fig.~\ref{Multi Geometry description}(b). Each layer $\Omega_j$ is occupied by a homogeneous material characterized by the dielectric permittivity $\epsilon_j$ and the magnetic permeability $\mu_j$. For reference, the planar layered medium case is shown in Fig.~\ref{Multi Geometry description}(a). Then the problem of Maxwell's equations reads: find the scattered fields $({\bm E}^{\rm sca}_j,{\bm H}^{\rm sca}_j)$, $j=1,\cdots,N$ satisfying
	\be
	\label{MultiOP}
	\begin{cases}
		\curl {\bm E}_j^{\rm sca}-i\omega\mu_j{\bm H}_j^{\rm sca}=\bm 0 & \mbox{in}\quad\Omega_j,\cr
		\curl {\bm H}_j^{\rm sca}+i\omega\epsilon_j{\bm E}_j^{\rm sca}=\bm 0 & \mbox{in}\quad\Omega_j,
	\end{cases}\quad j=1,\cdots,N,
	\en
	and the transmission conditions
	\be
	\label{MultiOP2}
	\begin{cases}
		{\bm\nu}\times {\bm E}_{j+1}^{\rm sca}-{\bm \nu}\times {\bm E}_{j}^{\rm sca}=\bm f_j &\mbox{on}\quad\Gamma_j,\cr
		{\bm \nu}\times {\bm H}_{j+1}^{\rm sca}-{\bm \nu}\times {\bm H}_{j}^{\rm sca}=\bm g_j &\mbox{on}\quad\Gamma_j,
	\end{cases}\quad j=1,\cdots,N-1,
	\en
	with the tangential vector fields $\bm f_j$ and $\bm g_j$ being given by
	\ben
	\bm f_j=\bm \nu\times {\bm E}_j^{\rm src}|_{\Gamma_j}-\bm{\nu}\times {\bm E}_{j+1}^{\rm src}|_{\Gamma_j},\quad \bm g_j=\bm{\nu}\times \bm H_j^{\rm src}|_{\Gamma_j}-\bm{\nu}\times \bm H_{j+1}^{\rm src}|_{\Gamma_j}.
	\enn
	Here, analogous to (\ref{auxpl}) and (\ref{auxpo}), ($\bm E_j^{\rm src}$,$\bm H_j^{\rm src}$) denote the auxiliary source, where $(\bm E_j^{\rm re},\bm H_j^{\rm re})$ represents the reflected field resulting from the scattering of the plane wave $(\bm E^{\rm inc}, \bm H^{\rm inc})$ by the planar layered medium shown in Fig.~\ref{Multi Geometry description}(a).
	
	Relying on the complex coordinate stretching in Section~\ref{sec3.1}, we can obtain the PML truncated problem as follows:
	\be
	\label{P}
	\begin{cases}
		\curl {\widetilde {\bm E}}_j^{\rm sca}-i\omega\mu_j\mathbb{A}^{-1}{\widetilde {\bm H}}_j^{\rm sca}=0 & \mbox{in}\quad\Omega_j^b,\quad j=1,...,N,\cr
		\curl {\widetilde {\bm H}}_j^{\rm sca}+i\omega\epsilon_j\mathbb {A}^{-1}{\widetilde {\bm E}}_j^{\rm sca}=0 & \mbox{in}\quad\Omega_j^b,\quad j=1,...,N,\cr
		\bm{\nu}\times {\widetilde {\bm E}}_{j+1}^{\rm sca}-\bm{\nu}\times {\widetilde {\bm E}}_{j}^{\rm sca}=\widetilde{\bm f}_j &\mbox{on}\quad\Gamma_j^b,\quad j=1,...,N-1,\cr
		\bm{\nu}\times {\widetilde {\bm H}}_{j+1}^{\rm sca}-\bm{\nu}\times {\widetilde {\bm H}}_{j}^{\rm sca}=\widetilde{\bm g}_j &\mbox{on}\quad\Gamma_j^b,\quad j=1,...,N-1,
	\end{cases}
	\en	
	where $\Omega_j^b$ and $\Gamma_j^b$ denote the truncated domains and interfaces, respectively, and the tangential vector fields $\widetilde{\bm f}_j$ and $\widetilde{\bm g}_j$ are given by
	\ben
	\widetilde{\bm f}_j=\bm {\nu}\times \widetilde{\bm E}_j^{\rm src}|_{\Gamma_j}-\bm{\nu}\times \widetilde{\bm E}_{j+1}^{\rm src}|_{\Gamma_j},\quad \widetilde{\bm g}_j=\bm{\nu}\times \widetilde{\bm H}_j^{\rm src}|_{\Gamma_j}-\bm{\nu}\times \widetilde{\bm H}_{j+1}^{\rm src}|_{\Gamma_j}.
	\enn
	Introducing the following layer potentials and BIOs
	\ben
	&&\widetilde S_j^i(\bm\varphi)(\bm x)=\mathbb{B}(\bm x)\widetilde \nabla_{\bm x}\times\widetilde \nabla_{\bm x}\times\int_{\Gamma_i^b} \widetilde\Phi(k_j,\bm x,\bm y)\mathbb{B}(\bm y)\bm\varphi(\bm y)ds_{\bm y},\quad \bm x\in\Omega_j^b,\\
	&&\widetilde D_j^i(\bm\varphi)(\bm x)=\mathbb{B}(\bm x)\widetilde \nabla_{\bm x}\times\int_{\Gamma_i^b} \widetilde\Phi(k_j,\bm x,\bm y)\mathbb{B}(\bm y)\bm\varphi(\bm y)ds_{\bm y},\quad \bm x\in\Omega_j^b,\\
	&&\widetilde N_j^{l,i}(\bm\varphi)(\bm x)=\bm {\nu_x}\times \mathbb{B}(\bm x)\widetilde \nabla_{\bm x}\times\widetilde \nabla_{\bm x}\times\int_{\Gamma_i^b} \widetilde\Phi(k_j,\bm x,\bm y)\mathbb{B}(\bm y)\bm\varphi(\bm y)ds_{\bm y},\quad \bm x\in\Gamma_l^b,\\
	&&\widetilde K_j^{l,i}(\bm\varphi)(\bm x)=\bm{\nu_x}\times \mathbb{B}(\bm x)\widetilde \nabla_{\bm x}\times\int_{\Gamma_i^b} \widetilde\Phi(k_j,\bm x,\bm y)\mathbb{B}(\bm y)\bm\varphi(\bm y)ds_{\bm y},\quad \bm x\in\Gamma_l^b,
	\enn	
	and assuming the vanishing of the electromagnetic fields $({\widetilde {\bm E}}_{j}^{\rm sca}, {\widetilde {\bm H}}_{j}^{\rm sca})$ on the outer boundary of the PML region, the fields within the layers can be expressed in the form
	\ben
	\nonumber
	&&\widetilde {\bm E}^{\rm sca}_1=-\widetilde D_1^1(\widetilde {\bm M}_1)-\frac{i}{\omega\epsilon_1}\widetilde S_1^1(\widetilde {\bm J}_1),\quad \widetilde {\bm H}_1^{\rm sca}=\frac{i}{\omega\mu_1}\widetilde S_1^1(\widetilde {\bm M}_1)-\widetilde D_1^1(\widetilde {\bm J}_1)\quad \mbox{in}\quad\Omega_1^b,\\
	\nonumber
	&&\widetilde {\bm E}_j^{\rm sca}=-\widetilde D_j^j(\widetilde {\bm M}_j)-\frac{i}{\omega\epsilon_j}\widetilde S_j^j(\widetilde {\bm J}_j)+\widetilde D_j^{j-1}(\widetilde {\bm M}_{j-1})+\frac{i}{\omega\epsilon_j}\widetilde S_j^{j-1}(\widetilde {\bm J}_{j-1})\quad \mbox{in}\quad\Omega_j^b,\\
	&&\widetilde {\bm H}_j^{\rm sca}=\frac{i}{\omega\mu_j}\widetilde S_j^j(\widetilde {\bm M}_j)-\widetilde D_j^j(\widetilde {\bm J}_j)-\frac{i}{\omega\mu_j}\widetilde S_j^{j-1}(\widetilde {\bm M}_{j-1})+\widetilde D_j^{j-1}(\widetilde {\bm J}_{j-1})\quad \mbox{in}\quad\Omega_j^b,\\
	\nonumber
	&&\widetilde {\bm E}_N^{\rm sca}=\widetilde D_N^{N-1}(\widetilde {\bm M}_{N-1})+\frac{i}{\omega\epsilon_N}\widetilde S_N^{N-1}(\widetilde {\bm J}_{N-1})\quad \mbox{in}\quad\Omega_N^b,\\
	\nonumber
	&& \widetilde {\bm H}_N^{\rm sca}=-\frac{i}{\omega\mu_N}\widetilde S_N^{N-1}(\widetilde {\bm M}_{N-1})+\widetilde D_N^{N-1}(\widetilde {\bm J}_{N-1})\quad \mbox{in}\quad\Omega_N^b.
	\enn
	for $j=2,\cdots,N-1$ in terms of the currents density functions
	\ben
	\widetilde {\bm M}_j=\bm\nu\times \widetilde {\bm E}^{\rm sca}_{j+1}|_{\Gamma_j}, \quad \widetilde {\bm J}_j=\bm\nu\times \widetilde {\bm H}^{\rm sca}_{j+1}|_{\Gamma_j},\quad j=1,...,N-1.
	\enn
	Utilizing the jump relations shown in Lemma~\ref{lemma1}, we are led to the following BIEs
	\be
	\label{MulPMLBIE1}
	&&\mathbb E_1 (\widetilde{\bm\phi}_1)+\widetilde{\mathbb T}_1(\widetilde{\bm\phi}_1)+\widetilde{\mathbb R}_1(\widetilde{\bm\phi}_2)=\widetilde {\bm F}_1 \quad\mbox{on}\quad\Gamma_1^b,\\
	\label{MulPMLBIEj}
	&&\mathbb E_j (\widetilde{\bm\phi}_j)+\widetilde{\mathbb L}_j(\widetilde{\bm\phi}_{j-1})+\widetilde{\mathbb T}_j(\widetilde{\bm\phi}_j)+\widetilde{\mathbb R}_j(\widetilde{\bm\phi}_{j+1})=\widetilde {\bm F}_j\mbox{on}\;\Gamma_j^b,\; j=2,...,N-2,\\
	\label{MulPMLBIEN}
	&&\mathbb E_{N-1} (\widetilde{\bm\phi}_{N-1})+\widetilde{\mathbb L}_{N-1}(\widetilde{\bm\phi}_{N-2})+\widetilde{\mathbb T}_{N-1}(\widetilde{\bm\phi}_{N-1})=\widetilde {\bm F}_{N-1}\quad\mbox{on}\quad\Gamma_{N-1}^b,
	\en
	where
	\ben
	\mathbb E_j=\begin{bmatrix}
		\frac{\epsilon_j+\epsilon_{j+1}}{2} &0\\
		0 &\frac{\mu_j+\mu_{j+1}}{2}
	\end{bmatrix}, \quad {\mathbb T}_j=\begin{bmatrix}
		\epsilon_j\widetilde K^{j,j}_j-\epsilon_{j+1}\widetilde K^{j,j}_{j+1} &\frac{i}{\omega}(\widetilde N_j^{j,j}-\widetilde N_{j+1}^{j,j})\\
		\frac{i}{\omega}(\widetilde N_{j+1}^{j,j}-\widetilde N_{j}^{j,j}) & \mu_j\widetilde K^{j,j}_j-\mu_{j+1}\widetilde K^{j,j}_{j+1}
	\end{bmatrix},
	\enn
	\ben
	{\mathbb R}_j=\begin{bmatrix}
		\epsilon_{j+1} \widetilde K^{j,j+1}_{j+1} &\frac{i}{\omega} \widetilde N_{j+1}^{j,j+1}\\
		-\frac{i}{\omega} \widetilde N_{j+1}^{j,j+1} & \mu_{j+1} \widetilde K^{j,j+1}_{j+1}
	\end{bmatrix},\quad {\mathbb L}_j=\begin{bmatrix}
		-\epsilon_{j} \widetilde K^{j,j-1}_{j} &-\frac{i}{\omega}\widetilde N_{j}^{j,j-1}\\
		\frac{i}{\omega}\widetilde N_{j}^{j,j-1} & -\mu_{j}\widetilde K^{j,j-1}_{j}
	\end{bmatrix}.
	\enn
	and
	\ben
	\widetilde{\bm\phi}_j=\begin{pmatrix}
		\widetilde {\bm M}_j\\
		\widetilde {\bm J}_j
	\end{pmatrix},\quad \widetilde {\bm F}_j=\begin{pmatrix}
		\epsilon_j(\bm{\nu}\times \widetilde {\bm E}_j^{\rm src}-\bm{\nu}\times \widetilde {\bm E}_{j+1}^{\rm src})|_{\Gamma_j}\\
		\mu_j(\bm{\nu}\times \widetilde {\bm H}_j^{\rm src}-\bm{\nu}\times \widetilde {\bm H}_{j+1}^{\rm src})|_{\Gamma_j}
	\end{pmatrix}.
	\enn	

	\section{Numerical experiments}
	\label{sec5}

	\subsection{Implementation strategy}
	\label{sec5.1}
	
	Making use of the regularization formulations presented in Section~\ref{sec3.4}, it follows that the numerical evaluation of all the BIOs in (\ref{PMLBIE}), (\ref{PMLPECBIE})-(\ref{PMLPMCBIE}) and (\ref{MulPMLBIE1})-(\ref{MulPMLBIEN}) can be degenerated into the discretization of weakly-singular integrals as well as the surface differential operators $\overrightarrow{\curl}_\Gamma$ and $\divv_{\Gamma}$. In this work, the Chebyshev-based rectangular-polar solver proposed in~\cite{BG20,BY20} will be utilized and we refer to \cite{LXYZ23} for the discretization of of the PML-transformed weakly-singular integrals and the operator $\overrightarrow{\curl}_\Gamma$. Here, we only supplement the evaluation of the surface divergence operator $\divv_{\Gamma}$.
	
	In terms of a non-overlapping parameterized (logically-rectangular) partition
	\ben
	\Gamma=\bigcup_{q=1}^M \Gamma_q, \quad \Gamma_q=\left\{ \bm{x}^q(u,v)=(x_1^q(u,v),x_2^q(u,v),x_3^q(u,v))^\top:\left[-1,1\right]^2\to \R^3 \right\}
	\enn
	of the surface $\Gamma^b$, each tangential field $\bm\varphi$ on $\Gamma_q$ admits a representation
	\ben
	\bm\varphi(\bm x^q)=\varphi_1(\bm x^q)\frac{\partial \bm{x}^q}{\partial u}+\varphi_2(\bm x^q)\frac{\partial \bm{x}^q}{\partial v},\quad \bm x^q\in \Gamma_q.
	\enn
	which further implies that
	\be
	\label{divG}
	(\divv_\Gamma \bm\varphi)(\bm x^q)=\frac{1}{\left|\mathbb G\right|}\left\{\frac{\partial}{\partial u}(\sqrt{\left|\mathbb G\right|}\varphi_1)+\frac{\partial}{\partial v}(\sqrt{\left|\mathbb G\right|}\varphi_2)\right\},
	\en
	where $\left|\mathbb G\right|$ denotes the determinant of the metric tensor $\mathbb G$ given by
	\ben
	\mathbb G(\bm x^q)=\begin{pmatrix}
		\frac{\partial \bm x^q}{\partial u}\cdot \frac{\partial \bm x^q}{\partial u}&\frac{\partial \bm x^q}{\partial u}\cdot \frac{\partial \bm x^q}{\partial v}\\
		\frac{\partial \bm x^q}{\partial v}\cdot \frac{\partial \bm x^q}{\partial u}&\frac{\partial \bm x^q}{\partial v}\cdot \frac{\partial \bm x^q}{\partial v}
	\end{pmatrix}, \quad \bm x^q\in\Gamma_q.
	\enn
	Given a fixed integer $N_p>0$ and the discretization points $\bm x^q_{ij}=\bm x^q(u_i,v_j)$ with
	\ben
	u_i=\cos\left(\frac{2i+1}{2N_p}\pi\right),\quad v_j=\cos\left(\frac{2j+1}{2N_p}\pi\right),\quad i,j=0,...,N_p-1.
	\enn
	the density function $\bm\varphi$ can be approximated on $\Gamma_q$ by the Chebyshev expansion
	\ben
	\bm\varphi (\bm x) \approx \sum\limits_{i,j = 0}^{N_p - 1} (\varphi_1)_{ij}^q a_{ij}(u,v)\frac{\partial \bm{x}^q}{\partial u}+\sum\limits_{i,j = 0}^{N_p - 1} (\varphi_2)_{ij}^q a_{ij}(u,v)\frac{\partial \bm{x}^q}{\partial v}, \quad \bm x\in \Gamma_q,
	\enn
	where
	\ben
	a_{ij}(u,v)=\frac{1}{N_p^2}\sum^{N_p-1}_{m,n=0} \alpha_n\alpha_mT_n(u_i)T_m(v_j)T_n(u)T_m(v),\quad {\alpha _n} = \begin{cases}
		1, & n=0,\cr
		2, & n\neq 0,
	\end{cases}
	\enn
	and $(\varphi_1) _{ij}^q=\varphi_1({\bm x}^q_{ij})$, $(\varphi_2) _{ij}^q=\varphi_2({\bm x}^q_{ij})$. Hence, the term $\divv_\Gamma\bm\varphi$ at each discretization point $\bm x^q_{ij}$ can be approximated through
	\ben
	(\divv_\Gamma\bm\varphi)(\bm x^q_{ij})=\sum_{n,m=0}^{N_p-1}(B_1)_{ij,nm}^q(\varphi_1)_{n,m}^q +\sum_{n,m=0}^{N_p-1}(B_2)_{ij,nm}^q(\varphi_2)_{n,m}^q,
	\enn
	where
	\ben
	&\quad&(B_1)_{ij,nm}^q\\
	&&=\frac{1}{\left|\mathbb G\right|}\left(\left|\frac{\partial \bm x^q}{\partial v}\right|^2\frac{\partial \bm x^q}{\partial u}\cdot\frac{\partial^2 \bm x^q}{\partial u^2}+\left|\frac{\partial \bm x^q}{\partial u}\right|^2\frac{\partial \bm x^q}{\partial v}\cdot\frac{\partial^2 \bm x^q}{\partial u\partial v}\right)a_{nm}\Big|_{u=u_i,v=v_j}\\
	&&+\left[\frac{\partial a_{n,m}}{\partial u}-\frac{1}{\left|\mathbb G\right|}\frac{\partial \bm x^q}{\partial u}\cdot\frac{\partial \bm x^q}{\partial v}\left(\frac{\partial \bm x^q}{\partial v}\cdot\frac{\partial^2 \bm x^q}{\partial u^2}+\frac{\partial \bm x^q}{\partial u}\cdot\frac{\partial^2 \bm x^q}{\partial u\partial v}\right)a_{n,m}\right]\Big|_{u=u_i,v=v_j},\\
	&\quad&(B_2)_{ij,nm}^q\\
	&&=\frac{1}{\left|\mathbb G\right|}\left(\left|\frac{\partial \bm x^q}{\partial v}\right|^2\frac{\partial \bm x^q}{\partial u}\cdot\frac{\partial^2 \bm x^q}{\partial u\partial v}+\left|\frac{\partial \bm x^q}{\partial u}\right|^2\frac{\partial \bm x^q}{\partial v}\cdot\frac{\partial^2 \bm x^q}{\partial v^2}\right)a_{nm}\Big|_{u=u_i,v=v_j}\\
	&&+\left[\frac{\partial a_{n,m}}{\partial v}-\frac{1}{\left|\mathbb G\right|}\frac{\partial \bm x^q}{\partial u}\cdot\frac{\partial \bm x^q}{\partial v}\left(\frac{\partial \bm x^q}{\partial v}\cdot\frac{\partial^2 \bm x^q}{\partial u\partial v}+\frac{\partial \bm x^q}{\partial u}\cdot\frac{\partial^2 \bm x^q}{\partial v^2}\right)a_{n,m}\right]\Big|_{u=u_i,v=v_j}.
	\enn

	\subsection{Numerical examples}
	\label{sec5.2}
	
	This section will present several numerical examples to demonstrate the efficiency and accuracy of the proposed PML-BIE method for solving the layered-medium electromagnetic scattering problems. The resulted linear system is solved by means of the fully complex version of the iterative solver GMRES with residual tolerance $\epsilon_r=10^{-8}$. All the numerical results presented in this paper were obtained by means of Fortran implementations, parallelized using OpenMP. The relative maximum error is defined by
	\be
	\label{RE}
	\epsilon_{\infty}&:=\frac{\mbox{max}_{\bm x\in \Gamma_{\mathrm{test}}}{\left|\bm E^{\rm num}(\bm x)-\bm E^{\rm ref}(\bm x)\right|}}{\mbox{max}_{\bm x\in \Gamma_{\mathrm{test}}}{\left|\bm E^{\rm ref}(\bm x)\right|}},
	\en
	where $\bm E^{\rm ref}$ denotes the exact solutions, when available, or numerical solution based on a sufficiently fine discretization, and where $\Gamma_{\mathrm{test}}$ is a suitably selected square plane away from the interfaces. In all cases, unless otherwise specified, the parameters for PML are set to be $P=6$, $S=6$ and $T_i=2\lambda$,  $i=1,2$, where $\lambda=2\pi/k_1$ denotes the wavelength.
	
	\begin{figure}[htb]
		\centering
		\begin{tabular}{cc}
			\includegraphics[scale=0.15]{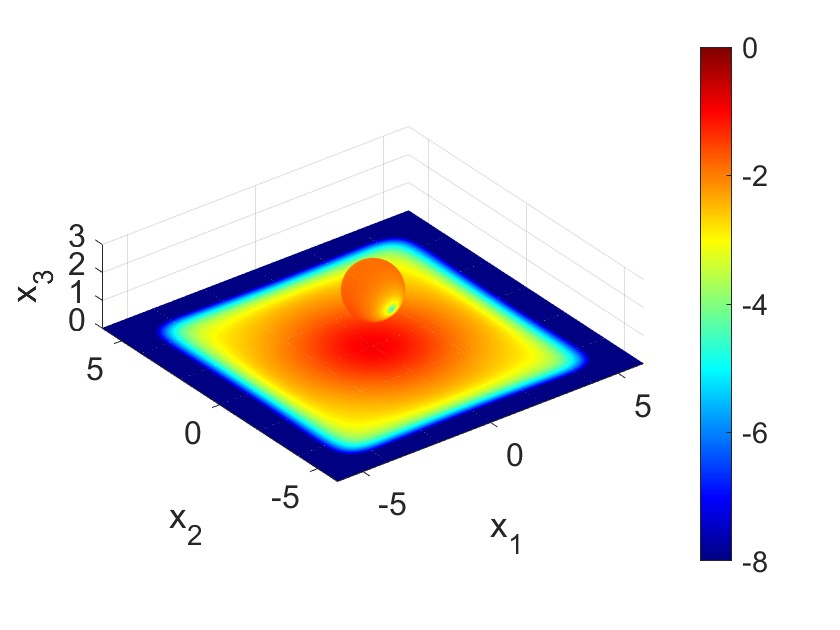} &
			\includegraphics[scale=0.15]{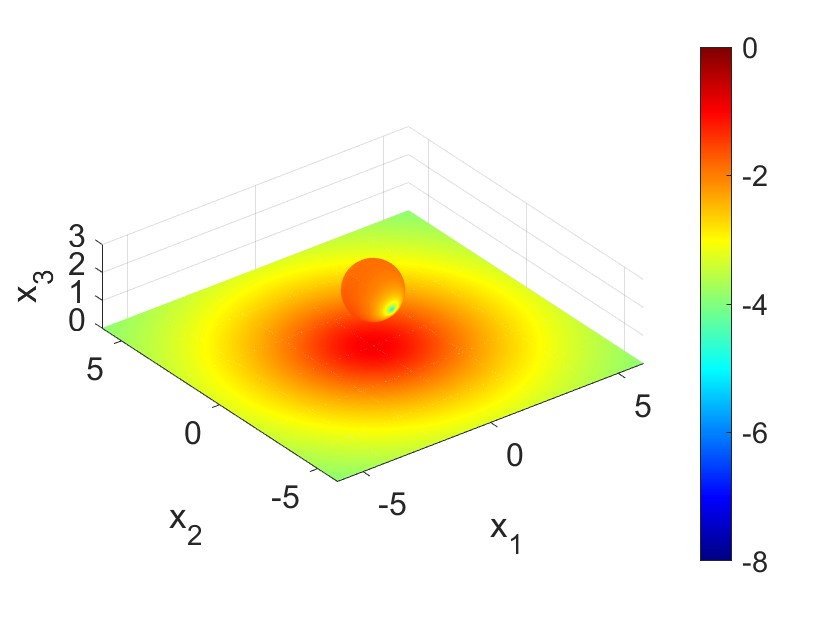} \\
			(a) $\log(\left|\widetilde {\bm J}\right|)$ &(b) $\log(\left|\bm J\right|)$ \\
			\includegraphics[scale=0.15]{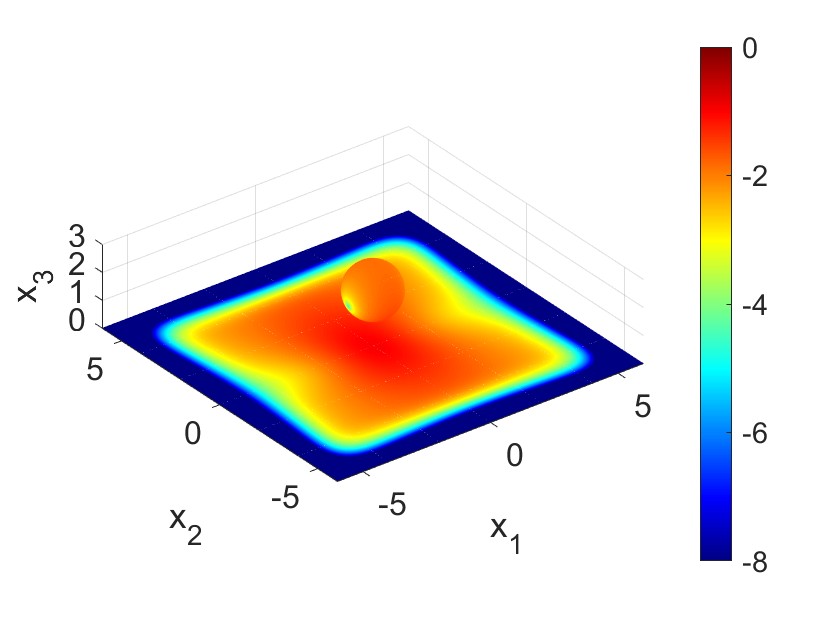} &
			\includegraphics[scale=0.15]{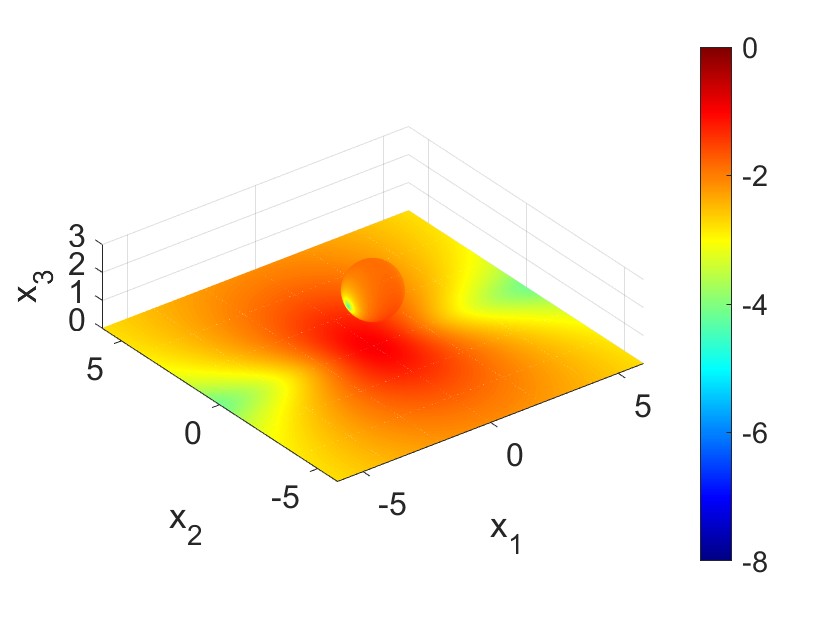} \\
			(c) $\log(\left|\widetilde {\bm M}\right|)$ &(d) $\log(\left|\bm M\right|)$ \\
		\end{tabular}
		\caption{Example 1. Absolute values of the numerical solution to the BIEs~(\ref{PMLPECBIE})-(\ref{PMLPMCBIE}) on $\Gamma_*^b$ as well as the exact values. (a)(b): PEC problem; (c)(d): PMC problem.}
		\label{Halfexample1}
	\end{figure}

	\begin{figure}[htb]
		\centering
		\begin{tabular}{cc}
			\includegraphics[scale=0.15]{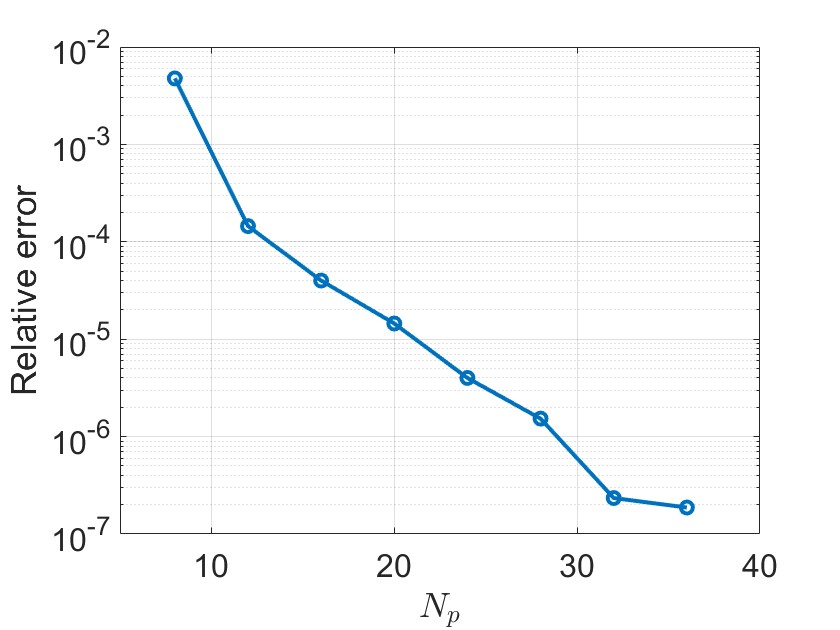} &
			\includegraphics[scale=0.15]{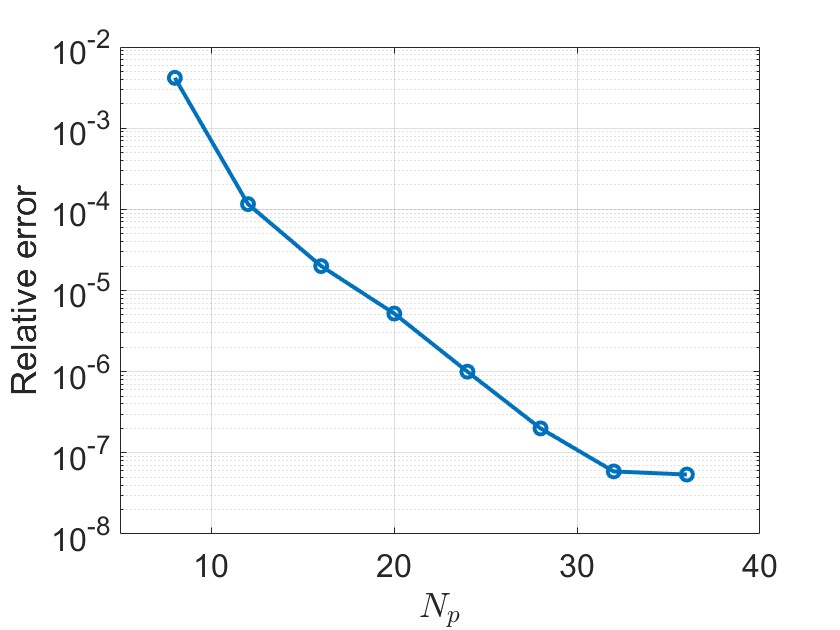} \\
			(a) PEC problem &(b) PMC problem
		\end{tabular}
		\caption{Example 1. Numerical errors $\epsilon_{\infty}$ for the PEC and PMC problems of scattering by a spherical obstacle on the half-space.}
		\label{Halfexample2}
	\end{figure}

	\begin{figure}[htb]
		\centering
		\begin{tabular}{cc}
			\includegraphics[scale=0.25]{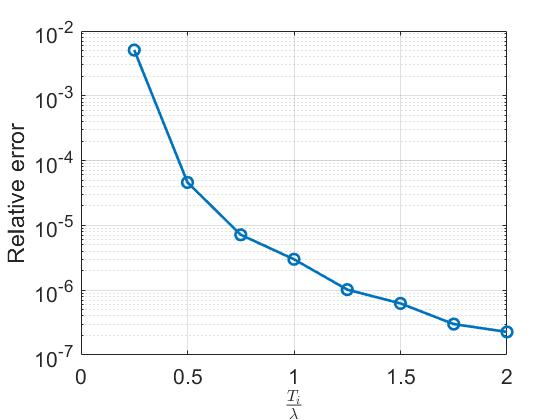} &
			\includegraphics[scale=0.25]{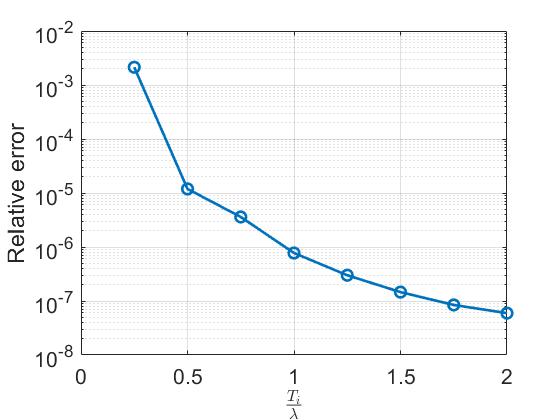} \\
			(a) PEC problem &(b) PMC problem
		\end{tabular}
		\caption{Example 1. Numerical errors $\epsilon_{\infty}$ with respect to $\frac{T_i}{\lambda}$.}
		\label{Halfexample3}
	\end{figure}
	
	{\bf Example 1.} We consider the problem of scattering by a spherical obstacle within a half space. We set $a_i=2$, $i=1,2$. The exact solutions ($\bm E^{\rm exa},\bm H^{\rm exa}$) are given by
	\ben
	\bm E^{\rm exa}(\bm x)=\nabla_{\bm x}\times\nabla_{\bm x}\times\begin{pmatrix}
		\Phi(k_1,\bm x,\bm z)\\
		0\\
		0
	\end{pmatrix},\quad \bm H^{\rm exa}(\bm x)=\frac{i}{\omega\epsilon_1}\nabla_{\bm x}\times\bm E^{\rm exa}(\bm x)
	\enn
	with $\bm z=(0,0,2)^\top$ located inside the obstacle. Fig.~\ref{Halfexample1} displays the numerical solutions of the BIEs (\ref{PMLPECBIE})-(\ref{PMLPMCBIE}) on $\Gamma_*^b$ with $\epsilon_1=1$, $\mu_1=1$ and $\omega=\pi$. It can be seen that the numerical solutions match perfectly with the exact solutions on $\Gamma_*^b\cap\ B_a$, and decay quickly on $\Gamma_*^b\cap B_{a,T}$. The relative errors $\epsilon_{\infty}$ with respect to different $N_p$ are depicted in Fig.~\ref{Halfexample2}, which clearly demonstrates the high accuracy and fast convergence of the PML-BIE solver. Letting $\omega=2\pi$, Fig.~\ref{Halfexample3} presents the relative errors with respect to $\frac{T_i}{\lambda}$ which shows that setting the PML thickness to be twice wavelength is sufficient to obtain high accuracy.

	\begin{figure}[htb]
		\centering
		\begin{tabular}{ccc}
			\includegraphics[scale=0.12]{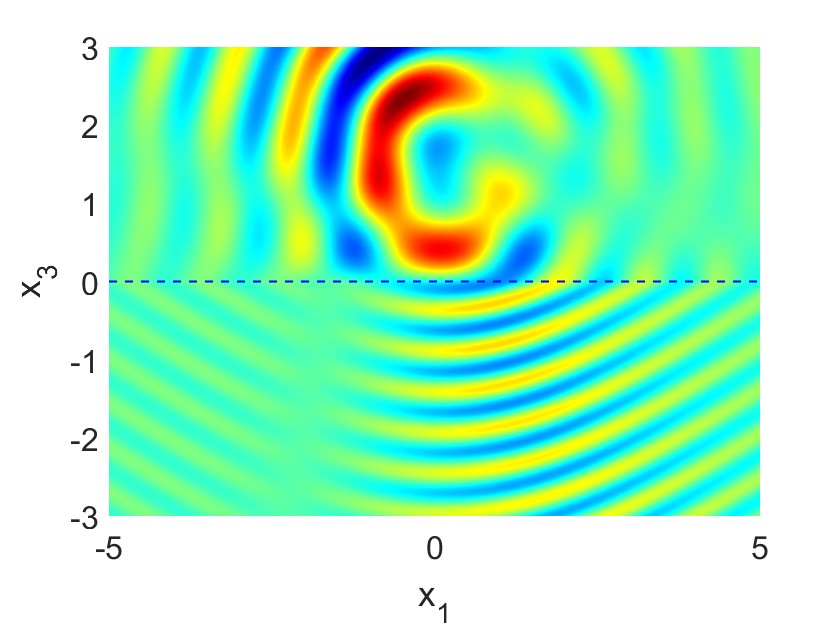} &
			\includegraphics[scale=0.12]{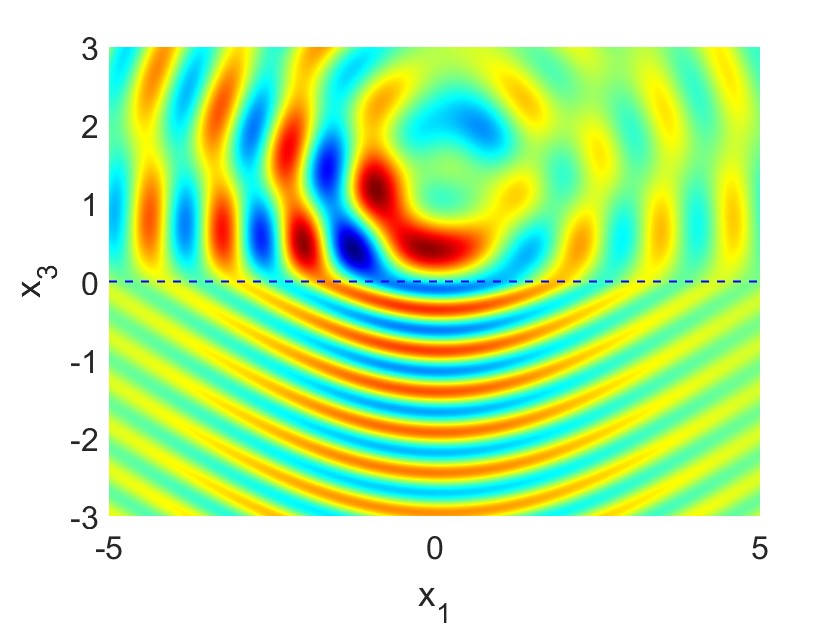} &
			\includegraphics[scale=0.12]{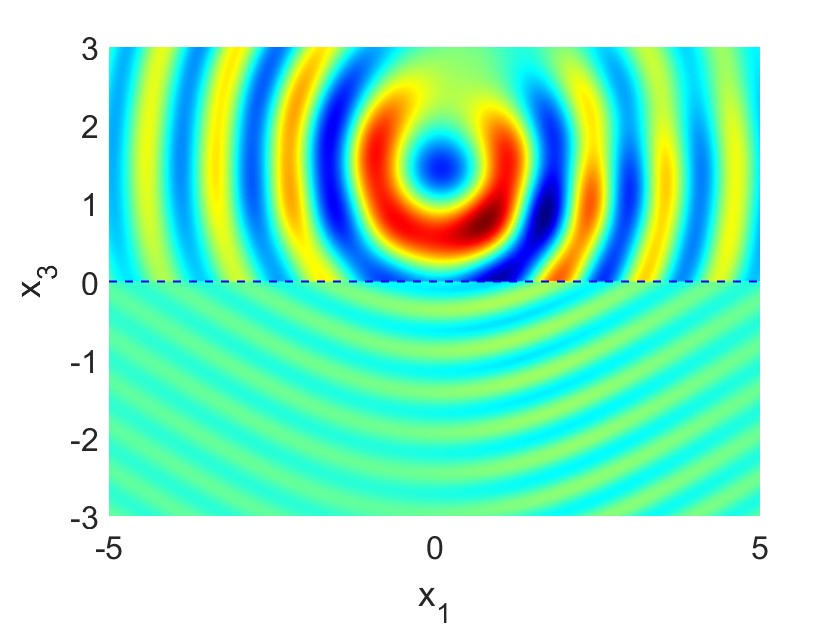} \\
			(a) $\real( E_{j,x_1})$  &(b) $\real( E_{j,x_2})$&(c) $\real( E_{j,x_3})$    \\
			\includegraphics[scale=0.12]{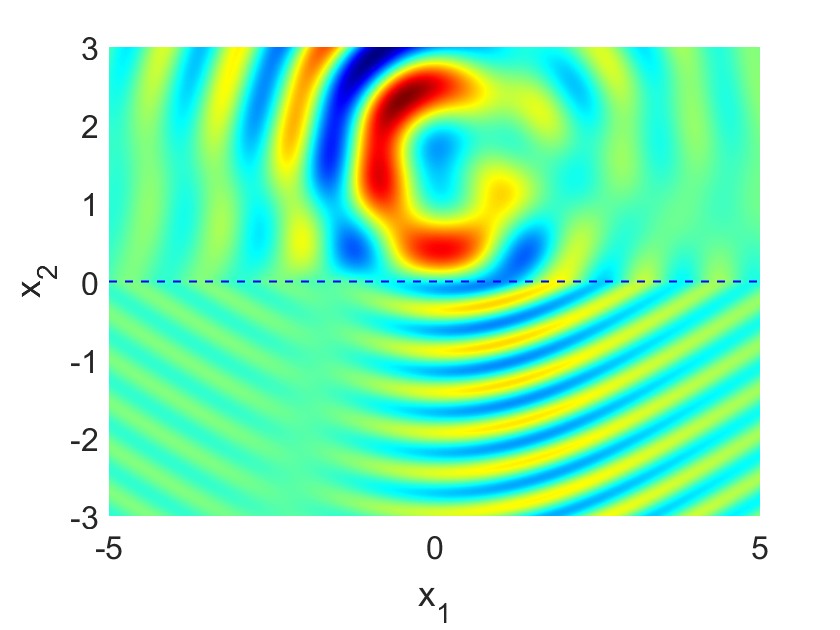} &
			\includegraphics[scale=0.12]{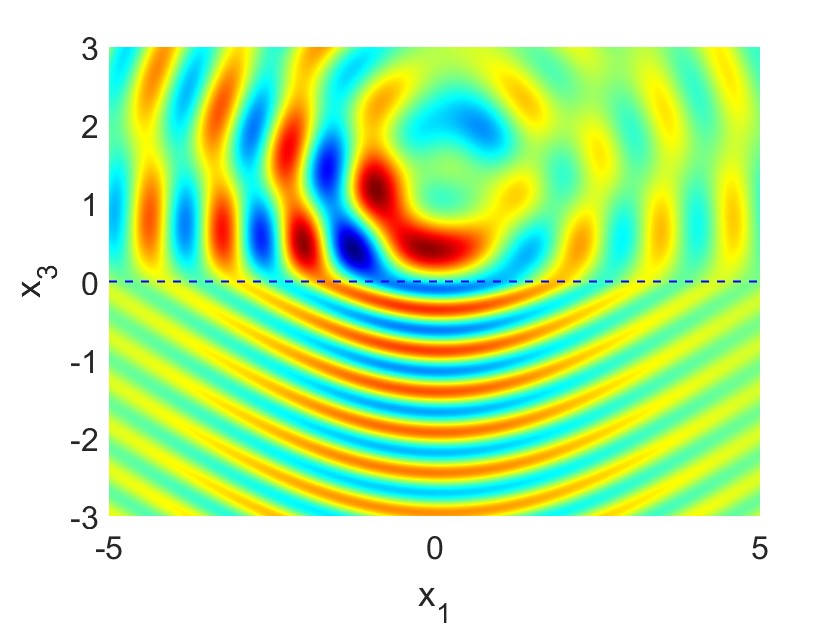} &
			\includegraphics[scale=0.12]{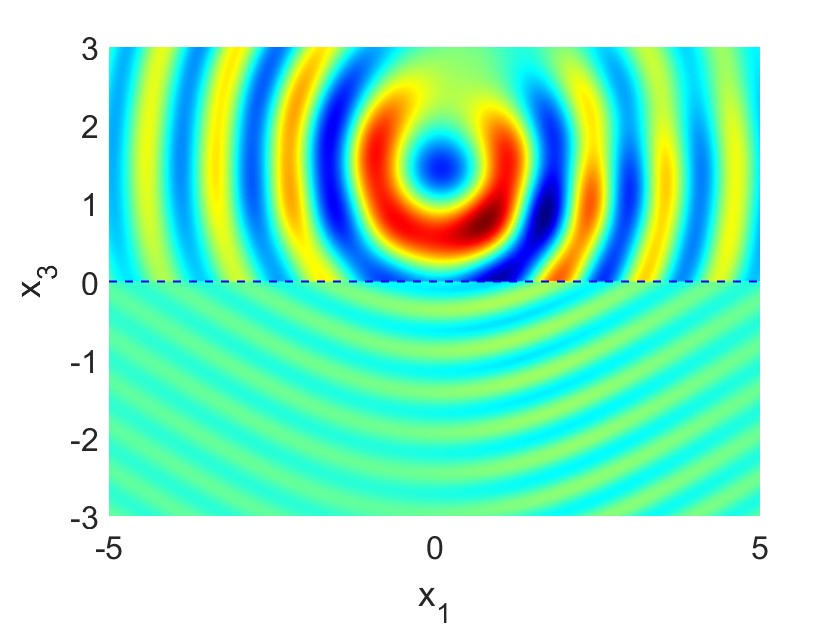} \\
			(d) $\real( E_{j,x_1})$  &(e) $\real( E_{j,x_2})$&(f) $\real( E_{j,x_3})$   \\
		\end{tabular}
		\caption{Example 2. Real parts of the three components of the total electric fields in a two-layered structure resulting from the PML-BIE solver (a,b,c) and the software Feko 2019 (d,e,f).}
		\label{MulExam1}
	\end{figure}
	
	\begin{figure}[htb]
		\centering
		\begin{tabular}{ccc}
			\includegraphics[scale=0.12]{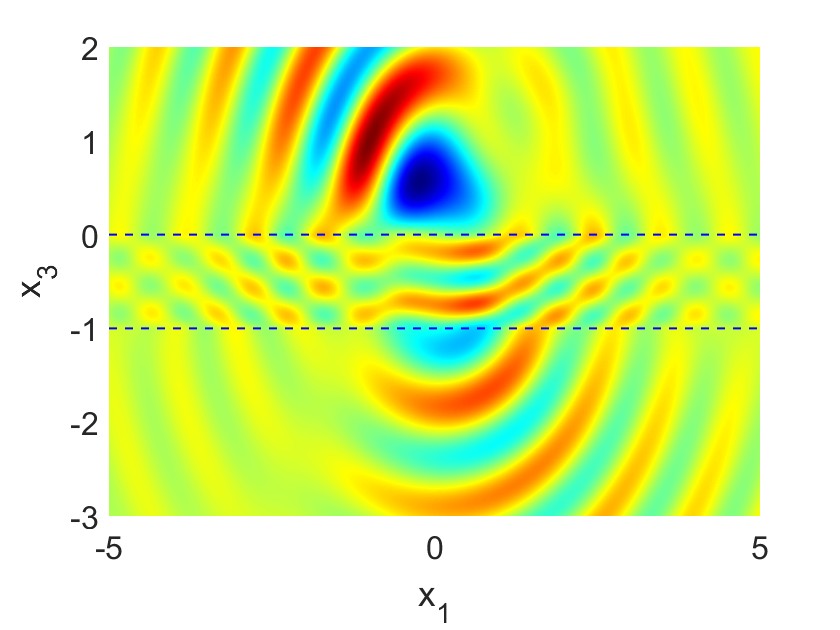} &
			\includegraphics[scale=0.12]{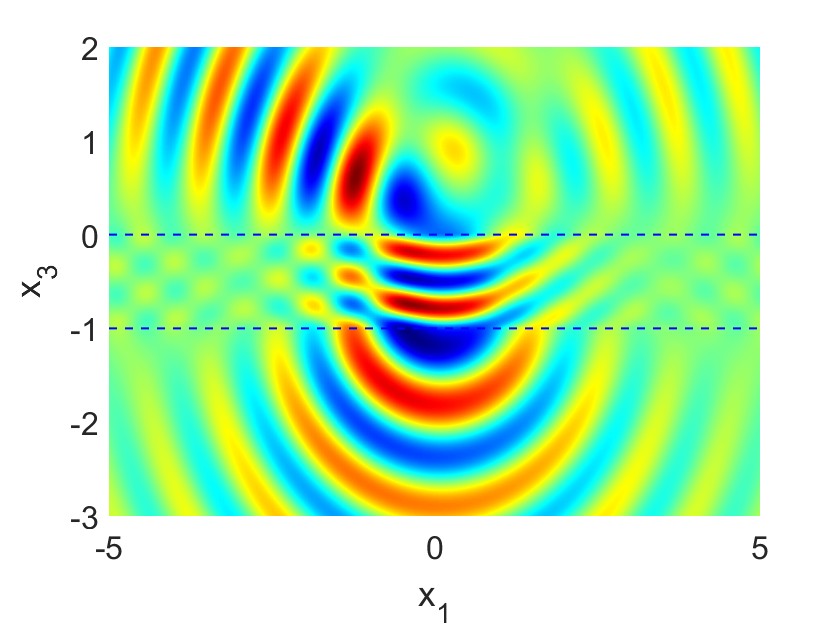} &
			\includegraphics[scale=0.12]{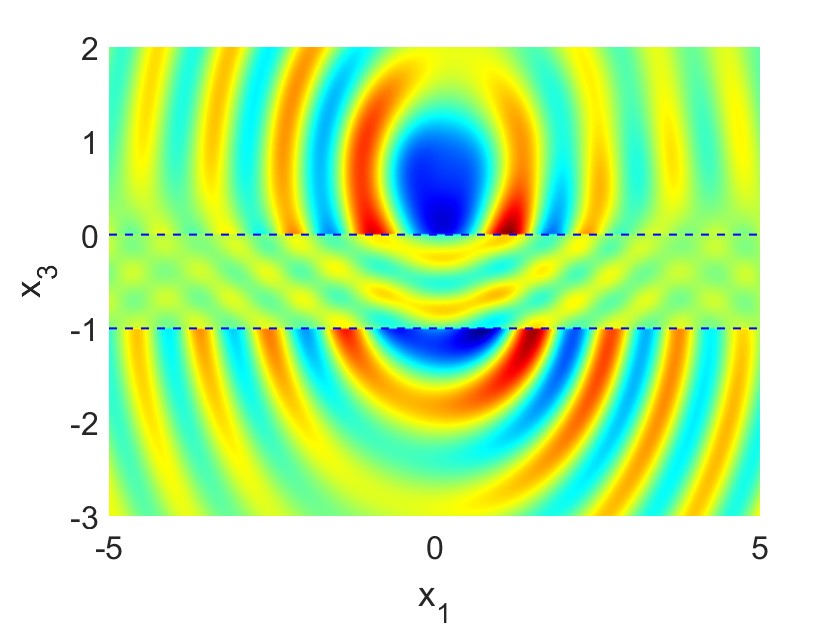} \\
			(a) $\real( E_{j,x_1})$  &(b) $\real( E_{j,x_2})$&(c) $\real( E_{j,x_3})$    \\
			\includegraphics[scale=0.12]{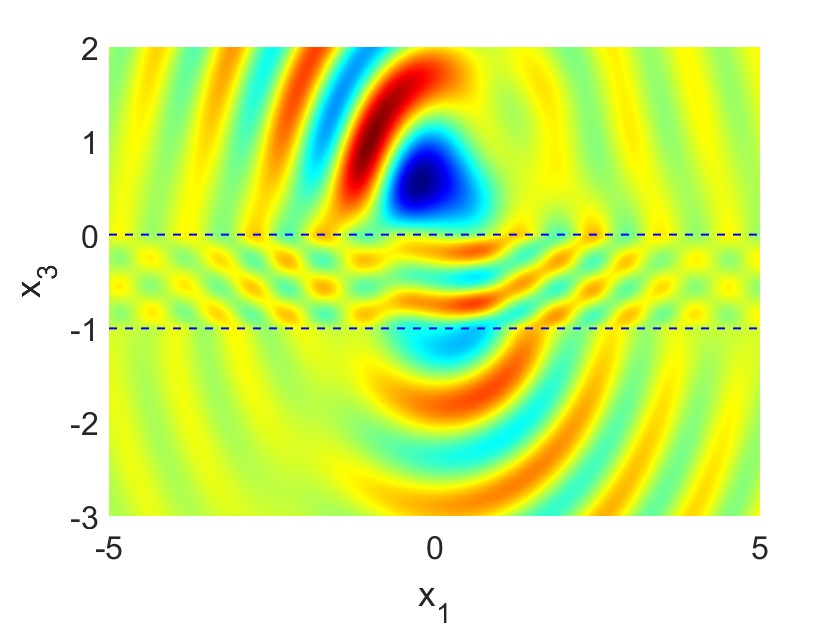} &
			\includegraphics[scale=0.12]{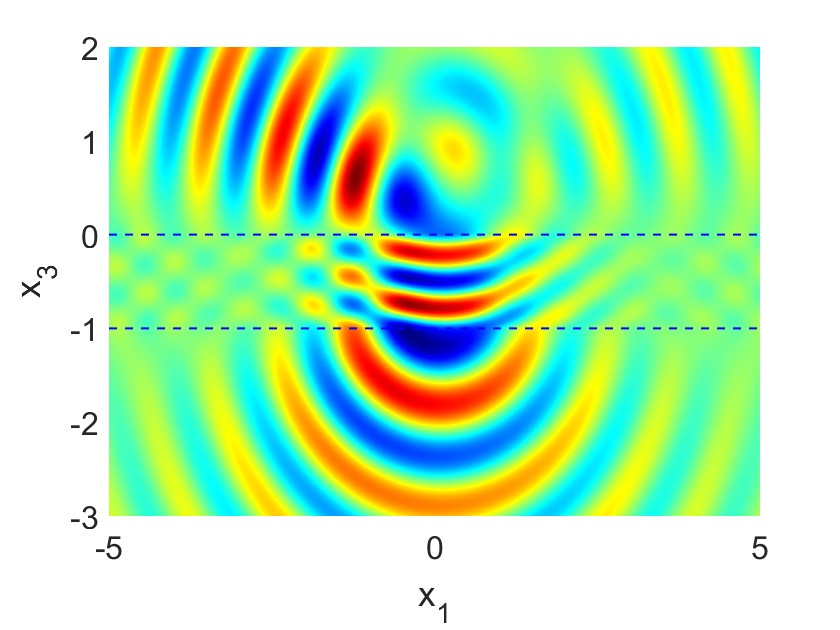} &
			\includegraphics[scale=0.12]{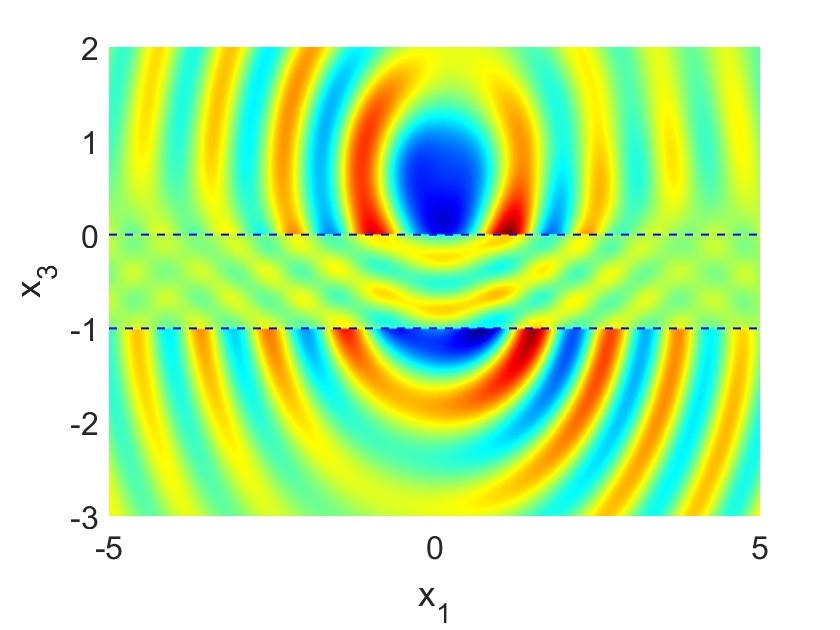} \\
			(d) $\real( E_{j,x_1})$  &(e) $\real( E_{j,x_2})$&(f) $\real( E_{j,x_3})$   \\
		\end{tabular}
		\caption{Example 2. Real parts of the three components of the total electric fields in a three-layered structure resulting from the PML-BIE solver (a,b,c) and the software Feko 2019 (d,e,f).}
		\label{MulExam2}
	\end{figure}

	\begin{figure}[htb]
		\centering
		\begin{tabular}{cc}
			\includegraphics[scale=0.15]{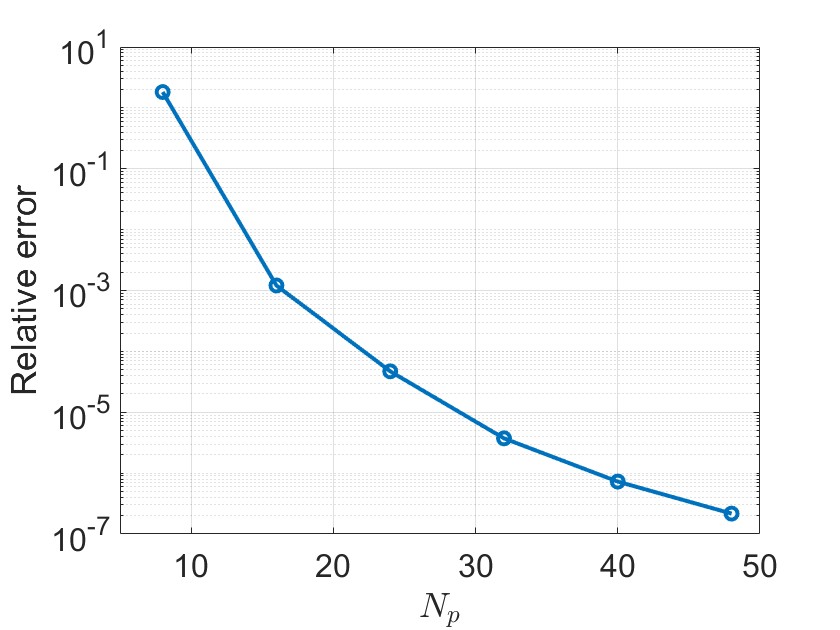} &
			\includegraphics[scale=0.15]{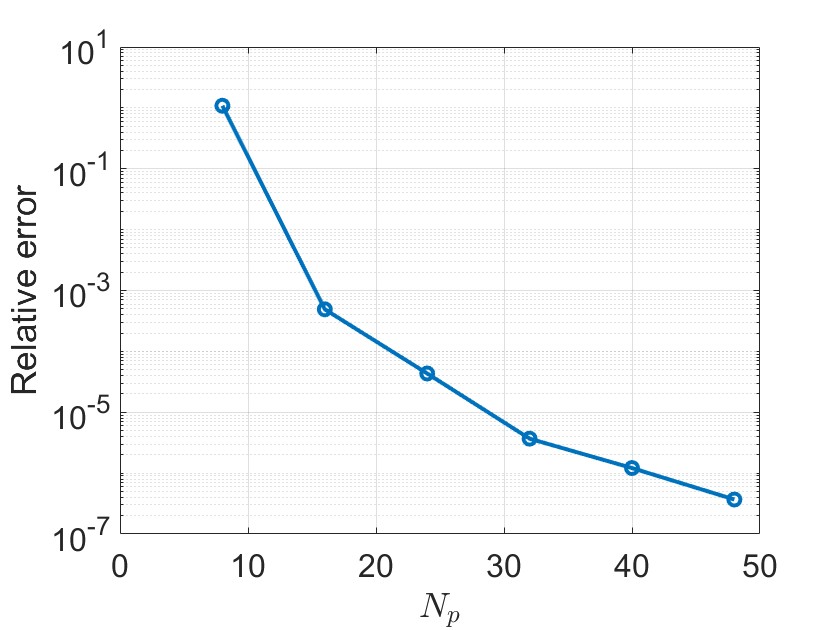} \\
			(a) Two-layered medium &(b) Three-layered medium    \\
		\end{tabular}
		\caption{Example 2. Numerical errors $\epsilon_{\infty}$ for the two- and three-layered medium problems with respect to $N_p$.}
		\label{MulExam4}
	\end{figure}
	
	{\bf Example 2.} In this example, we verify the efficiency of the proposed method for solving the multi-layered medium scattering problems. We first consider a two-layered medium problem with the interface $x_3=0$ and a dipole source (\ref{PSW}) located at $(0.1,-0.2,1.5)^\top$ with $\bm p=(\frac{1}{2},\frac{1}{2},\frac{1}{\sqrt 2})$. We choose $\epsilon_1=1$, $\epsilon_2=4$, $\mu_1=1$, $\mu_2=1$ and $\omega=2\pi$. Fig.~\ref{MulExam1} presents the real parts of the total electric fields at $x_2=1.2$ produced by the PML-BIE method, which are consistent with the fields resulting from the software Feko 2019. Next, we consider a three-layered structure which consists of two interfaces at $x_3=0$ and $x=-1$ and the dipole source is placed at $(0.1,-0.2,0.5)^\top$ with $\bm p=(\frac{1}{2},\frac{1}{2},\frac{1}{\sqrt 2})$. We set $\epsilon_1=1$, $\epsilon_2=4$, $\epsilon_3=1.1$, $\mu_1=1$, $\mu_2=1$, $\mu_3=1$ and $\omega=2\pi$. The total electric fields at $x_2=1$ resulting from the PML-BIE method and the software Feko 2019 are presented in Fig.~\ref{MulExam2} which demonstrate the efficiency of the proposed method.
	The relative errors for different values $N_p$ are depicted in Fig.~\ref{MulExam4} which clearly demonstrate high accuracy and rapid convergence of the proposed approach.
	
	\begin{figure}[htb]
		\centering
		\begin{tabular}{cc}
			\includegraphics[scale=0.15]{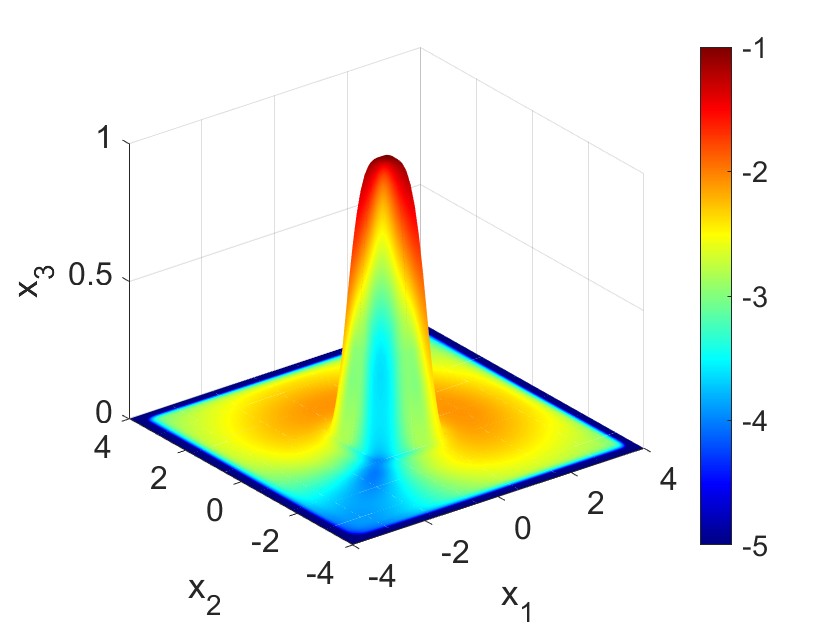} &
			\includegraphics[scale=0.15]{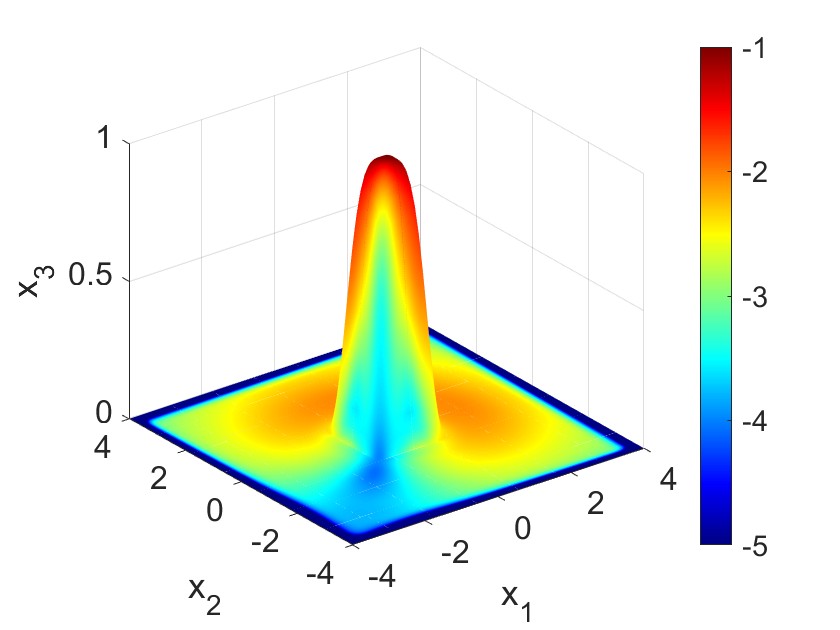} \\
			(a) $\log(|\widetilde{\bm M}|)$ &(b) $\log(|\widetilde{\bm J}|)$   \\
		\end{tabular}
		\caption{Example 3. Absolute values of the numerical solutions to the BIE (\ref{PMLBIE}) on $\Gamma^b$.}
		\label{MulExam5}
	\end{figure}

	\begin{figure}[htb]
		\centering
		\begin{tabular}{ccc}
			\includegraphics[scale=0.12]{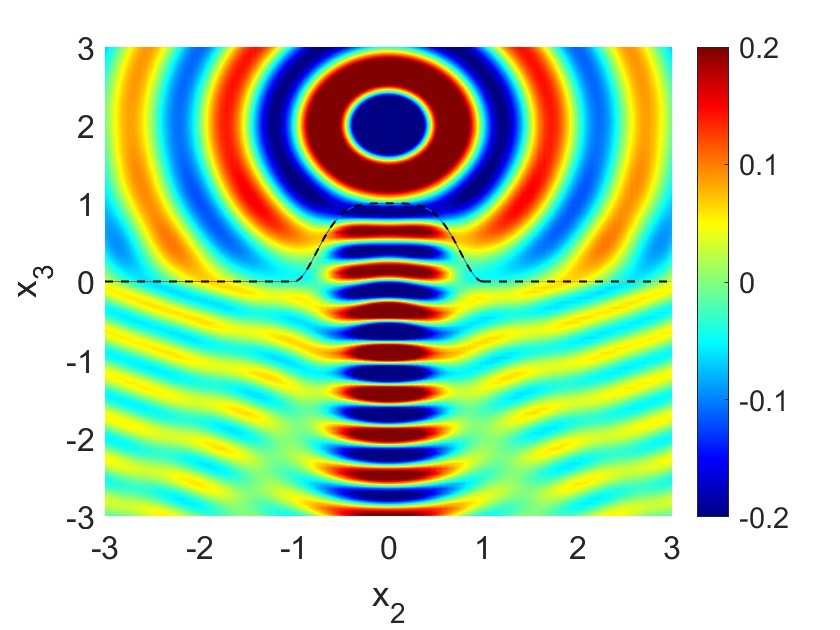} &
			\includegraphics[scale=0.12]{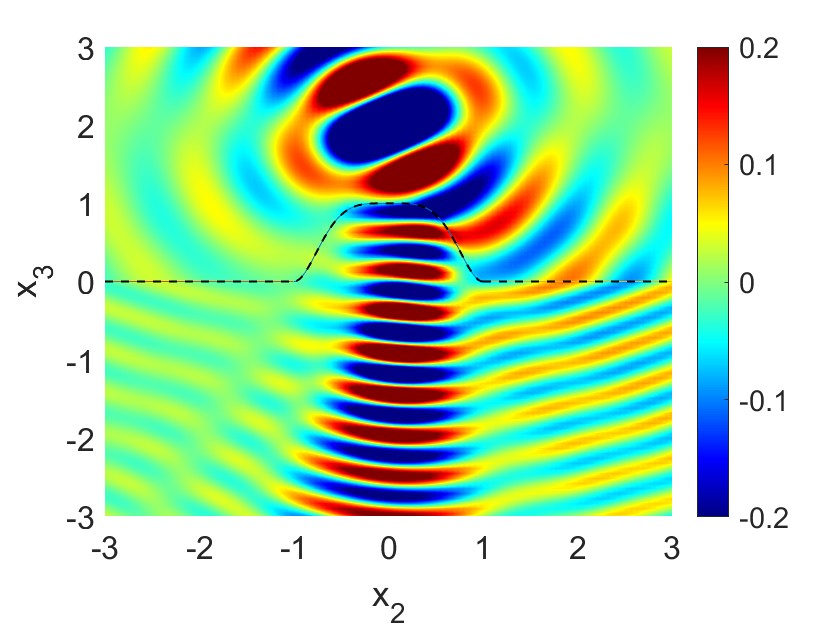} &
			\includegraphics[scale=0.12]{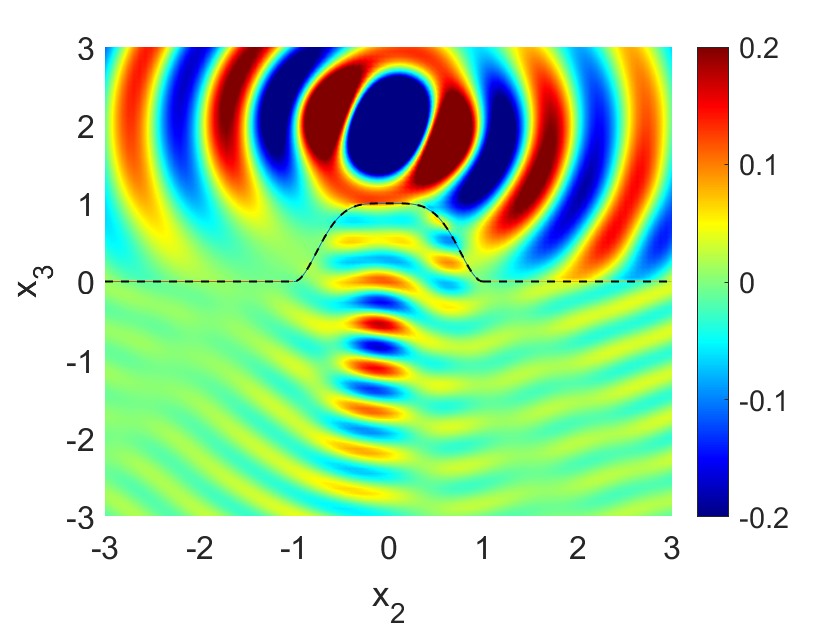} \\
			(a) $\real ( E_{j,x_1})$ &(b) $\real ( E_{j,x_2})$&(c) $\real ( E_{j,x_3})$   \\
			\includegraphics[scale=0.12]{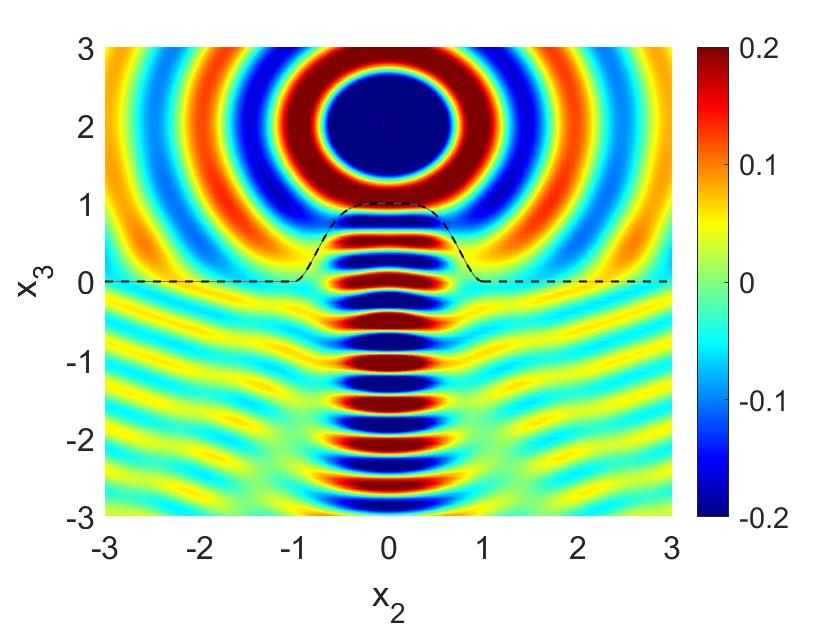} &
			\includegraphics[scale=0.12]{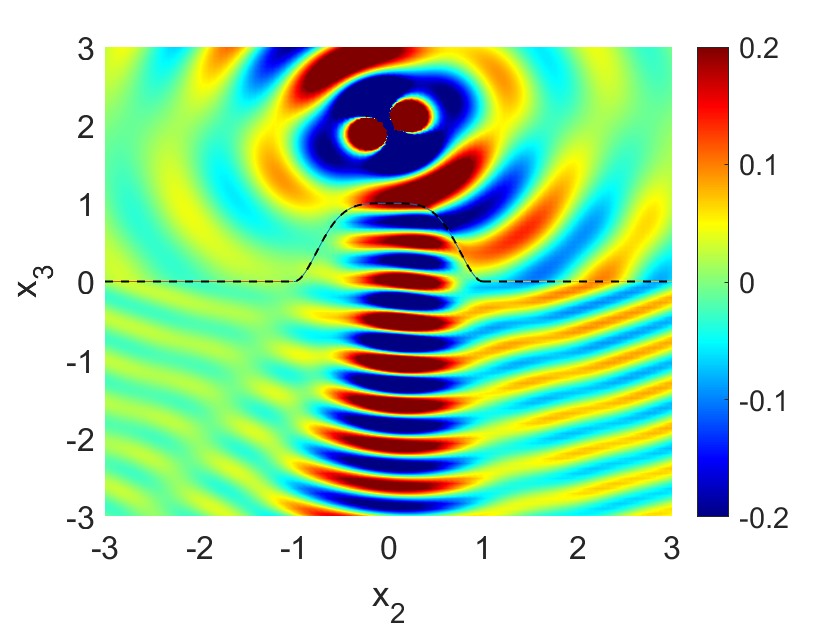} &
			\includegraphics[scale=0.12]{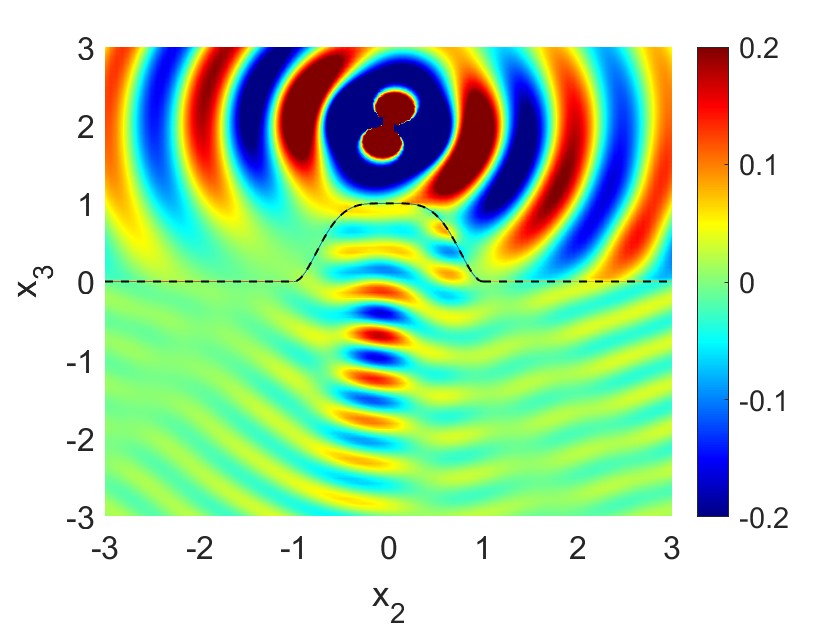} \\
			(d) $\Ima ( E_{j,x_1})$ &(e) $\Ima ( E_{j,x_2})$&(f) $\Ima ( E_{j,x_3})$   \\
		\end{tabular}
		\caption{Example 3. Real and imaginary parts of the three components of the total electric fields at $x_1=0$ for the two-layered medium problem.}
		\label{MulExam6}
	\end{figure}

	\begin{figure}[htb]
		\centering
		\begin{tabular}{ccc}
			\includegraphics[scale=0.15]{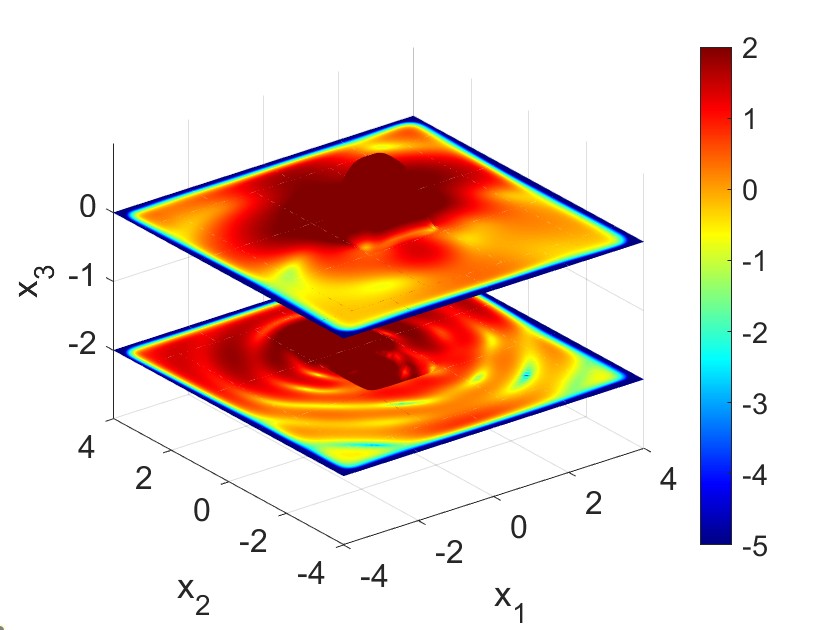} &
			\includegraphics[scale=0.15]{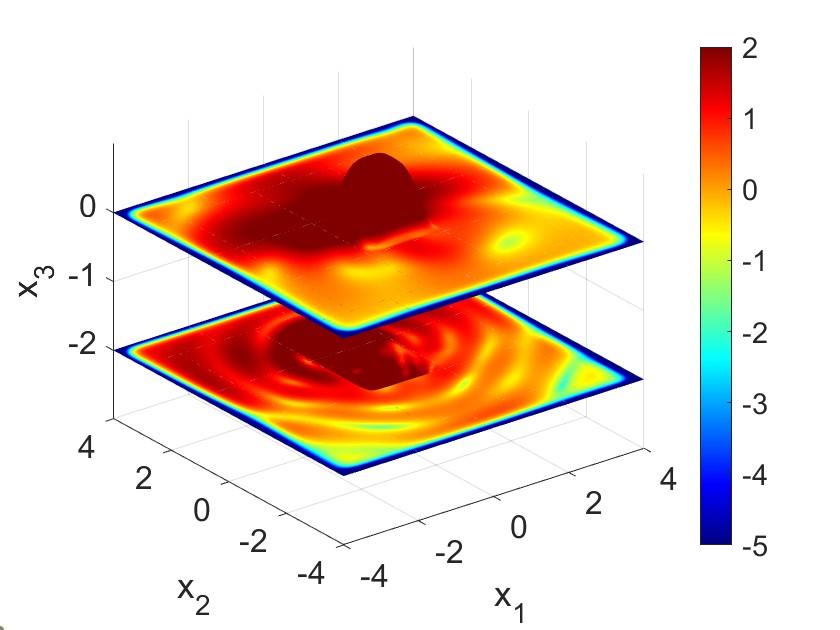} \\
			(a) $\log(|\widetilde {\bm M}|)$ & (b) $\log(|\widetilde {\bm J}|)$
		\end{tabular}
		\caption{Example 3. Absolute values of the numerical solutions to the BIEs (\ref{MulPMLBIE1})-(\ref{MulPMLBIEN}) on $\Gamma_j^b$, $j=1,2$.}
		\label{MulExam7}
	\end{figure}
	
	\begin{figure}[htb]
		\centering
		\begin{tabular}{ccc}
			\includegraphics[scale=0.12]{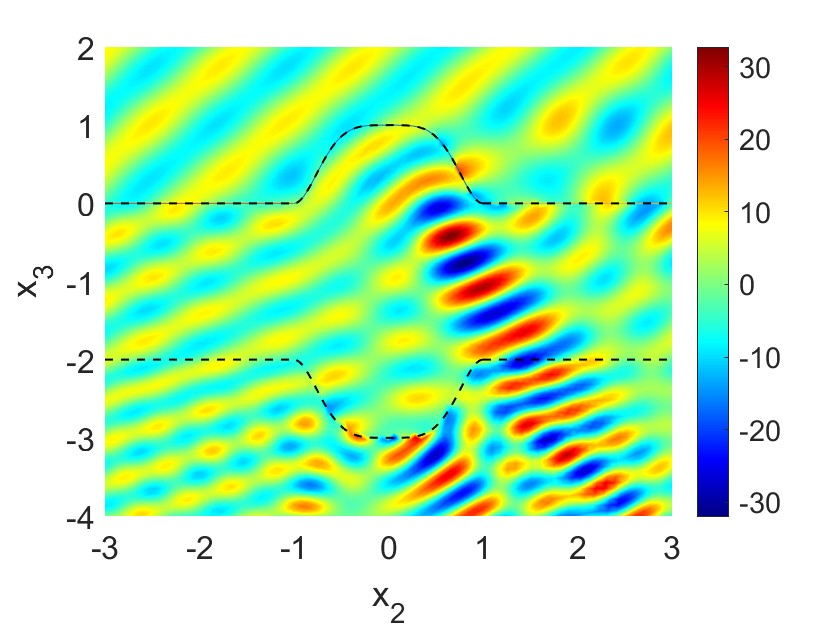} &
			\includegraphics[scale=0.12]{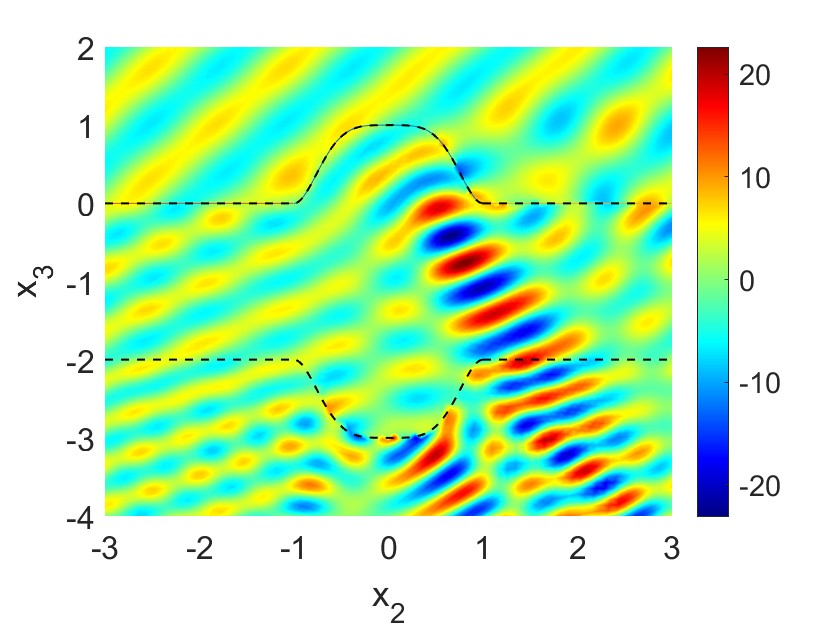} \\
			(a) $\real ( E_{j,x_1})$  & (b) $\real (H_{j,x_1})$\\
			\includegraphics[scale=0.12]{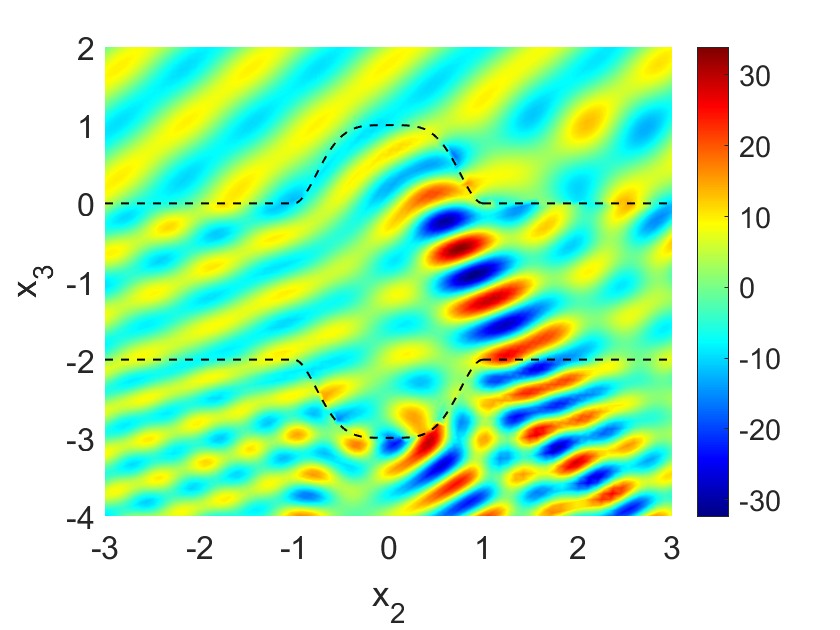} &
			\includegraphics[scale=0.12]{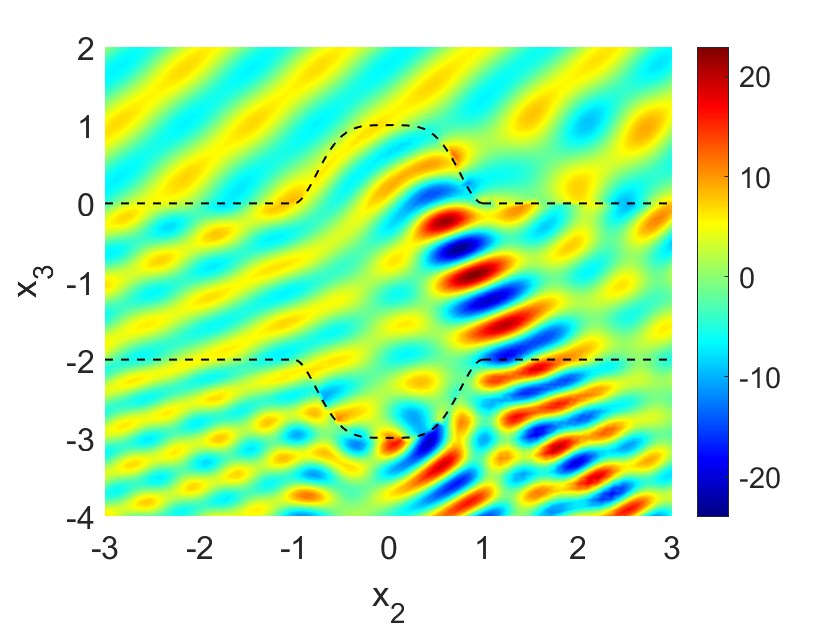} \\
			(c)  $\Ima ( E_{j,x_1})$ &(d) $\Ima ( H_{j,x_1})$   \\
		\end{tabular}
		\caption{Example 3. Real and imaginary parts of the first component of the total electromagnetic fields at $x_1=0$.}
		\label{MulExam8}
	\end{figure}

	{\bf Example 3.} Now we consider the problems of electromagnetic scattering by "locally-rough surfaces" in both two- and three-layered medium. We first study a two-layered medium problem where the local perturbation is characterized by $x_3=0.25(\cos(x_1^2\pi)+1)(\cos(x_2^2\pi)+1)$, $x_1,x_2\in [-1,1]$ and the medium parameters are selected to be $\epsilon_1=1$, $\epsilon_2=2$, $\mu_1=1$, $\mu_2=2$ and $\omega=2\pi$. As shown in Fig.~\ref{MulExam5}, the solution of BIEs (\ref{PMLBIE}) resulting from the PML-BIE method decay rapidly on the PML interface $\Gamma_{\rm{PML}}$, which clearly illustrates the effectiveness of the proposed method. Values of the three components of the total electric fields at $x_1=0$ are depicted in Fig.~\ref{MulExam6}. Next, we consider the problem of scattering of a plane wave by local perturbed surfaces in a three-layered medium with $\epsilon_1=2$, $\epsilon_2=3$, $\epsilon_3=4$, $\mu_1=2$, $\mu_2=3$, $\mu_3=4$ and $\omega=\pi$. The perturbed interfaces coincide with the graph of the functions $x_3=0.25(\cos(x_1^2\pi)+1)(\cos(x_2^2\pi)+1)$, $x_1,x_2\in [-1,1]$ and $x_3=-0.25(\cos(x_1^2\pi)+1)(\cos(x_2^2\pi)+1)-2$, $x_1,x_2\in [-1,1]$. The solution of the BIEs (\ref{MulPMLBIE1})-(\ref{MulPMLBIEN}) resulting from the PML-BIE method are displayed in the Fig.~\ref{MulExam7} which decay rapidly on the PML interface. Finally, Fig.~\ref{MulExam8} presents the first component of $\bm E_j$ and $\bm H_j$, $j=1,2,3$ at $x_1=0$.
	
\section*{Acknowledgments}
G. Bao was partially supported by the Key Project of Joint Funds for Regional Innovation and Development (U21A20425) and a Key Laboratory of Zhejiang Province. W. Lu was partially supported by NSFC Grant 12174310, and a Key Project of Joint Funds For Regional Innovation and Development (U21A20425). T. Yin gratefully acknowledges support from NSFC through Grants 12171465 and 12288201. L. Zhang was partially supported by NSFC Grant 12071060 and Postdoctoral Fellowship Program of CPSF GZC20232336.
	
	\appendix
	\section*{Appendix.}
	\renewcommand{\theequation}{A.\arabic{equation}}
	
	This appendix is arranged to present the expression of the fields $(\bm E_j^{\rm src},\bm H_j^{\rm src})$ for the case of plane wave incidence. Consider the scattering problem of a plane electromagnetic wave (\ref{PPW}) by a planar layered medium with $N\ge 2$. The specific derivations are based on the waves-tracing arguments, see for example~\cite{B12,C95}. In detail, introducing the propagation constants $k_{j,x_3}=\sqrt{k_j^2-k_{1,x_2}^2}$, $j=2,...,N$ with the complex square root defined such that $\Ima k_{j,x_3}\ge 0$, we can write the $x_1-$component of the planar-medium solution~\cite{C17} in $\Omega_j$, $j=1,...,N$, as
	\ben
	&& E_{j,x_1}^{\rm src}(\bm x)=E_0e^{ik_{1,x_2}x_2}\begin{cases}
		e^{-ik_{1,x_3}x_3}+\widetilde R_{12}^{\rm TE}e^{ik_{1,x_3}(x_3+2d_1)}, & j=1,\cr
		A_j^{\rm TE}\left[e^{-ik_{j,x_3}x_3}+\widetilde R_{j,j+1}^{\rm TE}e^{ik_{j,x_3}(x_3+2d_j)}\right] , &  2\le j\le N,
	\end{cases}\\
	&& H_{j,x_1}^{\rm src}(\bm x)=H_0e^{ik_{1,x_2}x_2}\begin{cases}
		e^{-ik_{1,x_3}\bm x_3}+\widetilde R_{12}^{\rm TM}e^{ik_{1,x_3}(x_3+2d_1)}, & j=1,\cr
		A_j^{\rm TM}\left[e^{-ik_{j,x_3}x_3}+\widetilde R_{j,j+1}^{\rm TM}e^{ik_{j,x_3}(x_3+2d_j)}\right] , & 2\le j\le N,
	\end{cases}
	\enn
	relying on the amplitudes
	\ben
	E_0=-p_{x_3}k_{1,x_2}-p_{x_2}k_{1,x_3},\quad H_0=\frac{k_1^2}{\omega\mu_1}p_{x_2},
	\enn
	the amplitudes $A_j^{\rm TE}$, $A_j^{\rm TM}$ and generalized reflection coefficients $\widetilde R_{j,j+1}^{\rm TE}$ and $\widetilde R_{j,j+1}^{\rm TM}$. The amplitudes and the generalized reflection coefficients admit the following recursive relations
	\ben
	&&A_j^{\rm TE}=\begin{cases}
		1,& j=1,\cr
		\frac{T_{j-1,j}^{\rm TE}A_{j-1,j}^{\rm TE}e^{i(k_{j-1,x_3}-k_{j,x_3})d_{j-1}}}{1-R_{j,j-1}^{\rm TE}\widetilde R_{j,j+1}^{\rm TE}e^{2ik_{j,x_3}(d_j-d_{j-1})}},& j=2,...,N,
	\end{cases},\\
	&&A_j^{\rm TM}=\begin{cases}
		1,& j=1,\cr
		\frac{T_{j-1,j}^{\rm TM}A_{j-1,j}^{\rm TM}e^{i(k_{j-1,x_3}-k_{j,x_3})d_{j-1}}}{1-R_{j,j-1}^{\rm TM}\widetilde R_{j,j+1}^{\rm TM}e^{2ik_{j,x_3}(d_j-d_{j-1})}},& j=2,...,N,
	\end{cases}
	\enn
	and
	\ben
	&&\widetilde R_{j-1,j}^{\rm TE}=\begin{cases}
		0,& j=N+1,\cr
		R_{j-1,j}^{\rm TE}+\frac{T_{j,j-1}^{\rm TE}\widetilde R_{j,j+1}^{\rm TE}T_{j-1,j}^{\rm TE}e^{2ik_{j,x_3}(d_j-d_{j-1})}}{1-R_{j,j-1}^{\rm TE}\widetilde R_{j,j+1}^{\rm TE}e^{2ik_{j,x_3}(d_j-d_{j-1})}},& j=2,...,N,
	\end{cases},\\
	&&\widetilde R_{j-1,j}^{\rm TM}=\begin{cases}
		0,& j=N+1,\cr
		R_{j-1,j}^{\rm TM}+\frac{T_{j,j-1}^{\rm TM}\widetilde R_{j,j+1}^{\rm TM}T_{j-1,j}^{\rm TM}e^{2ik_{j,x_3}(d_j-d_{j-1})}}{1-R_{j,j-1}^{\rm TM}\widetilde R_{j,j+1}^{\rm TM}e^{2ik_{j,x_3}(d_j-d_{j-1})}},& j=2,...,N,
	\end{cases}
	\enn
	with the reflection coefficients
	\ben
	R_{j,j+1}^{\rm TE}=\frac{\mu_{j+1}k_{j,x_3}-\mu_{j}k_{j+1,x_3}}{\mu_{j+1}k_{j,x_3}+\mu_{j}k_{j+1,x_3}},\quad R_{j,j+1}^{\rm TM}=\frac{\epsilon_{j+1}k_{j,x_3}-\epsilon_{j}k_{j+1,x_3}}{\epsilon_{j+1}k_{j,x_3}+\epsilon_{j}k_{j+1,x_3}}
	\enn
	and the transmission coefficients
	\ben
	T_{j,j+1}^{\rm TE}=\frac{2\mu_{j+1}k_{j,x_3}}{\mu_{j+1}k_{j,x_3}+\mu_{j}k_{j+1,x_3}},\quad T_{j,j+1}^{\rm TM}=\frac{2\epsilon_{j+1}k_{j,x_3}}{\epsilon_{j+1}k_{j,x_3}+\epsilon_{j}k_{j+1,x_3}}.
	\enn
	Then we can obtain the other two components of the $\bm E_j^{\rm src}$ and $\bm H_j^{\rm src}$, $j=1,...,N$ given by~\cite{AP22}
	\ben
	&&E_{j,x_2}^{\rm src}=\frac{i}{\omega\epsilon_{j}}\frac{\partial H_{j,x_1}^{\rm src}}{\partial x_3},\quad E_{j,x_3}^{\rm src}=-\frac{i}{\omega\epsilon_{j}}\frac{\partial H_{j,x_2}^{\rm src}}{\partial x_1}, \\
	&&H_{j,x_2}^{\rm src}=-\frac{i}{\omega\mu_{j}}\frac{\partial E_{j,x_1}^{\rm src}}{\partial x_3},\quad E_{j,x_3}^{\rm src}=\frac{i}{\omega\mu_{j}}\frac{\partial E_{j,x_1}^{\rm src}}{\partial x_2}.
	\enn


\begin{thebibliography}{00}
\bibitem{ALC23} A-S. Bonnet-Ben Dhia, L. Faria, C. P\'{e}rez-Arancibia, A complex-scaled boundary integral equation for time-harmonic water waves, arXiv:2310.04127.
		\bibitem{AP22} R. Arrieta, C. P\'{e}rez-Arancibia, Windowed {G}reen function {MoM} for second-kind
		surface integral equation formulations of layered media electromagnetic scattering problems, IEEE Trans. Antennas and Propagation. 70 (12) (2022) 11978-11989.
		\bibitem{BB21} C. Bauinger, O. P. Bruno, "Interpolated factored green function'' method for accelerated solution of scattering problems, J. Comput. Phys. 430 (2021) 110095.
		\bibitem{B12} L. Brekhovskikh, Waves in Layered Media, vol. 16, Elsevier, 2012.
		\bibitem{BG20} O.P. Bruno, E. Garza, A {C}hebyshev-based rectangular-polar integral solver for scattering by general geometries described by non-overlapping patches, J. Comput. Phys. 421 (2020) 109740.
		\bibitem{BLPT16} O.P. Bruno, M. Lyon, C. P\'erez-Arancibia, C. Turc, Windowed {G}reen function method for layered-media scattering, SIAM J. Appl. Math. 76 (2016) 1871-1898.
		\bibitem{BP17} O.P. Bruno, C. P\'erez-Arancibia, Windowed {G}reen function method for the {H}elmholtz equation in presence of multiply layered media, Proc. R. Soc. Lond. Ser. A Math. Phys. Eng. Sci. 473 (2202) (2017) 20170161.
		\bibitem{BY20} O.P. Bruno, T. Yin, Regularized integral equation methods for elastic scattering problems in three dimensions, J. Comput. Phys. 410 (2020) 109350.
		\bibitem{BY21} O. P. Bruno, T. Yin, A windowed {G}reen function method for elastic scattering problems on a half-space, Comput. Methods Appl. Mech. Engrg. 376 (2021) 113651.
		\bibitem{C17} C. P\'{e}rez-Arancibia, Windowed integral equation methods for problems of scattering by defects and obstacles in layered media (Ph.D.thesis), California Institute of Technology, 2016.
		\bibitem{CY00} W. Cai, T. J. Yu, Fast calculations of dyadic {G}reen's functions for electromagnetic scattering in a multiplayered medium, J. Comput. Phys. 5 (2000) 247-251.
		\bibitem{CZ10} Z. Chen, W. Zheng, Convergence of the uniaxial perfectly matched layer method for time-harmonic scattering problems in two-layered media, SIAM J. Numer. Anal. 48 (6) (2010) 2158-2185.
		\bibitem{CZ17} Z. Chen, W. Zheng, {PML} method for electromagnetic scattering problem in a two-layer medium, SIAM J. Numer. Anal. 55 (4) (2017) 2050-2084.
		\bibitem{C95} W. C. Chew, Waves and Fields in Inhomogeneous Media, vol. 522, IEEE Press, 1995.
		\bibitem{CK98} D. Colton, R. Kress, Inverse Acoustic and Electromagnetic Scattering Theory, Springer, Berlin, 1998.
		\bibitem{DM97} J. De Santo, P.A. Martin, On the derivation of boundary integral equations for scattering by an infinite one-dimensional rough surface, J. Acoust. Soc. Am. 102 (1997) 67-77.
		\bibitem{DJZ20} X. Duan, X. Jiang, W. Zheng, Exponential convergence of Cartesian PML method for Maxwell's equations in a two-layer medium, ESAIM Math. Model Numer. Anal. 54 (2020) 929-956.
		\bibitem{GL22} Y. Gao, W. Lu, Wave scattering in layered orthotropic media {I}: a stable {PML} and a high-accuracy boundary integral equation method, SIAM J. Sci. Comp. 44 (4) (2022) B861-B884.
		\bibitem{GJ23} L. Greengard, S. Jiang, A Dual-space multilevel kernel-splitting framework for discrete and continuous convolution, arXiv.2308.00292.
		\bibitem{GR87} L. Greengard, V. Rokhlin, A fast algorithm for particle simulations, J. Comput. Phys. 73(2) (1987) 325-348.
		\bibitem{HW08} G.C. Hsiao, W.L. Wendland, Boundary Integral Equations, Applied Mathematical Sciences, Springer-verlag, Berlin, 2008.
		\bibitem{K14} R. Kress, Linear Integral Equations, 3rd ed., Springer, New York, 2014.
		\bibitem{L21} W. Lu, Mathematical analysis of wave radiation by a step-like surface, SIAM J. on Applied. Math. 81 (2) (2021) 666-693.
		\bibitem{LLQ18} W. Lu, Y.Y. Lu, J. Qian, Perfectly matched layer boundary integral equation method for wave scattering in a layered medium, SIAM J. Appl. Math. 78 (1) (2018) 246-265.
		\bibitem{LS01} M. Lassas, E. Somersalo, Analysis of the {PML} equations in general convex geometry, Proc. Roy. Soc. Edinburgh Sect. A. 131 (2001) 1183-1207.
		\bibitem{LXYZ23} W. Lu, L. Xu, T. Yin, L. Zhang, A highly accurate perfectly-matched-layer boundary integral equation solver for acoustic layered-medium problems, SIAM J. Sci. Comput. 45 (4) (2023) B523–B543.
		\bibitem{MC96} C. Macaskill, P. Cao, A new treatment of rough surface scattering, Proc. R. Soc. Lond. A. 452 (1996) 2593-2612.
		\bibitem{MC01} A. Meier, S.N. Chandler-Wilde, On the stability and convergence of the finite section method for integral equation formulations of rough surface scattering, Math. Methods Appl. Sci. 24 (2001) 209-232.
		\bibitem{MSS14} D. Miret, G. Soriano, M. Saillard, Rigorous simulations of microwave scattering from finite conductivity two-dimensional sea surfaces at low grazing angles, IEEE Trans. Geosci. Remote Sensing. 52 (2014) 3150-3158.
		\bibitem{O04} M. Ochmann, The complex equivalent source method for sound propagation over an impedance plane, J. Acoustical Soc. Amer. 116 (2004) 3304-3311.
		\bibitem{OC04} V.I. Okhmatovski, A.C. Cangellaris, Evaluation of layered media {G}reen's functions via rational function fitting, IEEE Microw. Wirel. Components Lett. 14 (2004) 22-24.
		\bibitem{PGM00} M. Paulus, P. Gay-Balmaz, O. Martin, Accurate and efficient computation of the {G}reen's tensor for stratified media, Phys. Rev. E. 62 (2000) 5797-5807.
		\bibitem{SS11} M. Saillard, G. Soriano, Rough surface scattering at low-grazing incidence: a dedicated model, Radio Sci. 46 (2011) RS0E13.
		\bibitem{SS112} S. Sauter, C. Schwab, Boundary Element Methods, Springer, Berlin, 2011.
		\bibitem{SSS08} P. Spiga, G. Soriano, M. Saillard, Scattering of electromagnetic waves from rough surfaces: a boundary integral method for low-grazing angles, IEEE Trans. Antennas Propagation. 56 (2008) 2043-2050.
		\bibitem{TCPS02} L. Tsang, C.H. Chan, K. Pak, H. Sangani, Monte-{C}arlo simulations of large-scale problems of random rough surface scattering and applications to grazing incidence with the {BMIA}/canonical grid method, IEEE Trans. Antennas Propagation. 43 (8) (2002) 851-859.
		\bibitem{YHLR22} X. Yu, G. Hu, W. Lu, A. Rathsfeld, {PML} and high-accuracy boundary integral equation solver for wave scattering by a locally defected periodic surface, SIAM J. Numer. Anal. 60 (5) (2022) 2592-2625.
		\bibitem{ZLSC05} Z. Zhao, L. Li, J. Smith, L. Carin, Analysis of scattering from very large three-dimensional rough surfaces using {MLFMM} and ray-based analyses, IEEE Antennas Propagation Magazine. 47 (2005) 20-30.
\end{thebibliography}
\end{document}